\documentclass[a4paper,11pt]{amsart}

\usepackage{amsmath, amsthm}
\usepackage[applemac]{inputenc}
\usepackage{color}
\usepackage[T1]{fontenc}
\usepackage{amsfonts}
\usepackage{graphicx}
\usepackage{psfrag}
\newtheorem{theorem}{Theorem}[section]
\newtheorem{lemma}[theorem]{Lemma}
\newtheorem{proposition}[theorem]{Proposition}

\theoremstyle{definition}

\theoremstyle{remark}
 \newtheorem{remark}[theorem]{Remark}

\def \vseq {\vspace{2mm}}
\def \vv {{\bf v}}
\def \uu {{\bf u}}
\def \nn {{\bf n}}
\def \OO {\Omega}
\def \GG {\Gamma}
\def \UU {{\bf U}}
\def \setR { \mathbb{R}}
\def \virg {\,, \,\,}
\def \ONE {\mathbf{1}}
\def \eps {\varepsilon}
\def \bft {{\bf t}}
\def \bfi { \mathbf{ i} }
\def \DD { \ell}

\newcommand{\ve}{\varepsilon}

\newcommand{\Ov}{\overline{\Omega}}
\newcommand{\Oa}{\mathring{\Omega}}

\newcommand{\lcal}{\mathcal{L}}

\newcommand{\R}{\mathbb{R}}

\newcommand{\haus}{\mathcal{H}}

\DeclareMathOperator{\diam}{diam}

\DeclareMathOperator{\spt}{supp}
\newcommand\tro{\tilde{\rho}}
\def \qq {{\bf q } }
\def \uu {{\bf u }}
\def \UU {{\bf U }}

\begin{document}

\title[A MACROSCOPIC CROWD MOTION MODEL
OF  GRADIENT FLOW TYPE]{A MACROSCOPIC CROWD MOTION MODEL\\
OF  GRADIENT FLOW TYPE}


\author
{BERTRAND MAURY
}
\address{Laboratoire de math\'ematiques, Universit\'e Paris-Sud, \\
Bat. 425,    91405 Orsay Cedex, France\\
Bertrand.Maury@math.u-psud.fr
}

\author
{AUDE ROUDNEFF-CHUPIN
}
\address{Laboratoire de math\'ematiques, Universit\'e Paris-Sud, \\
Bat. 425,   91405 Orsay Cedex, France\\
Aude.Roudneff@math.u-psud.fr
}

\author
{FILIPPO SANTAMBROGIO
}
\address{Laboratoire CEREMADE, Universit\'e Paris-Dauphine,\\
Pl. de Lattre de Tassigny, 75775 Paris Cedex 16, France\\
filippo@ceremade.dauphine.fr
}


\maketitle

\begin{abstract}
 A simple model to handle the flow of people in emergency evacuation situations is considered: at every point $x$, the velocity  $\UU(x)$ that individuals at $x$ would like to realize is given. Yet, the incompressibility constraint prevents this velocity field to be realized and the actual velocity is the projection of the desired one onto the set of admissible velocities. Instead of looking at a microscopic setting (where individuals are represented by rigid discs), here the macroscopic approach is investigated, where the unknown is 
a density $\rho(t,x)$. If a gradient structure is given, say $\UU=-\nabla D$  where $D$ is, for instance, the distance to the exit door, the problem is presented as a Gradient Flow in the Wasserstein space of probability measures. The functional which gives the Gradient Flow is neither finitely valued (since it takes into account the constraints on the density), nor geodesically convex, which requires for an {\it ad-hoc} study of the convergence of a discrete scheme.
\end{abstract}



\section{Introduction}

In the last two decades, several strategies have been proposed to model the motion of pedestrians. 
Most of them rely on a microscopic approach: the degrees of  freedom are the positions of individuals, and  their evolution depends on a balance between selfish behaviour, congestion constraints, and possibily social factors (politeness, gregariousness). 
Among those microscopic models, some are based on a stochastic description of the individual behaviour (see e.g. \cite{Hend}), whereas others are purely deterministic (see \cite{Helb3,Hoog1,Hoog2}).  

An essential ingredient in those models  lies in  the way interactions between individuals are handled, in particular in the case of high density (congestion phenomena). 
Following the classification which holds in the  modelling of granular flows, one can  differentiate  the Molecular Dynamics  (MD) approach (the non overlapping constraint between rigid grains is relaxed, and handled by a short range repulsive force) and the Contact Dynamics (CD) one (the collisions are explicitely  taken into account). In the context of pedestrians,  MD strategy has proved to be quite efficient to model congestion. In particular Helbing\cite{Helb1,Helb3,Helb5} introduced the concept of social forces, which are designed in such a way that individuals tend to repel each other when their distance drops below a certain value.  The model proposed in \cite{crowd1}  
relies on the alternative strategy: individuals do not interact with each other as soon as they are not in contact, and the non overlapping constraint is treated in a strong (non  relaxed) way. 
Although it is natural to expect some link between the two approaches (MD models are likely  to converge in some way to their CD counterparts as the repulsive force stiffness goes to infinity), it is to be noticed that the mathematical structures of the two classes of models are quite different. In the first case, Cauchy-Lipschitz theory for ODE's applies, whereas CD models present some analogies with the so-called sweeping process introduced by Moreau\cite{Mor} in the 70',  for which  a dedicated framework has been developped (see \cite{Thi1}, \cite{Thi2}, \cite{crowd2}).

In the case of macroscopic models, the 
first strategy (congestion is treated in a relaxed way) is favoured,
as it allows to use classical methods for studying PDE. 
For example, crowd motion models inspired from traffic flow models have been developped (see Refs.~\cite{Col,Cha1,Cha2}). They take the form of hyperbolic conservation laws, and they are essentially monodimensional in space.
In higher dimension, Bellomo and Dogbe\cite{Bel,Dog} proposed second order models, where a phenomenological relation describes how the crowd modifies its own speed:
$$\left\{ \begin{array}{llrcl}
\partial_{t} \rho & + & \nabla_{\mathbf{x}} \cdot (\rho \mathbf{v}) & = & 0\\
\partial_{t} \mathbf{v} & + & (\mathbf{v} \cdot \nabla_{\mathbf{x}}) \mathbf{v} & = & \mathbf{F} (\rho, \mathbf{v}).
\end{array} \right.$$
Typically, the motion is governed by $\mathbf{F}$, which has two parts: a relaxation term toward a definite speed, and a repulsive term to take into account that pedestrians tend to avoid high density areas.
Degond\cite{Deg} uses the same approach to  model sheep herds. In this model, the term $\mathbf{F}$ depends on a pressure which blows up when the density approaches a given congestion density (barrier method).
There also exist first order models, where the velocity field is directly defined as a function of the density (see e.g. \cite{Hug1,Hug2,Cos}).
Another class of models is described by Piccoli and Tosin in Refs.~\cite{Pic1}, \cite{Pic2}. They propose a time-evolving measures framework, where the velocity of the pedestrian is composed by two terms: a desired velocity and an interaction velocity. The last one models the reaction of the pedestrian to the other surrounding pedestrians (namely, people can deviate from their preferred path if they enter a crowded area).

To our knowledge, as the ones presented above, all macroscopic models rely on a relaxed expression of the congestion. Let us mention however the work of Buttazzo, Jimenez and Oudet in \cite{But1}, where the optimal transportation between two given densities is computed under constraints (obstacles, congestion, ...) which can be strongly expressed. Yet, this approach is very different from the model we describe later, since its goal is to find an optimal transport between densities as in the work of Benamou and Brenier\cite{BenBre} (which is the classical reference for dynamical formulations of transport problems), whereas optimal transportation is in our case a very suitable tool. Moreover, we will mainly make use of the distance that optimal transport induces on probability measures  rather than looking at the optimal maps themselves, as we will see after a brief description of the model we consider.



The macroscopic model we present here is based on a strong expression of the congestion constraint.
It is a natural extension of  the microscopic  approach proposed in Refs.~\cite{crowd1,crowd2,crowd3}, which we describe here in its simpler form. The crowd configuration is represented by the position vector $\qq = (\qq_1,\dots,\qq_N)$. 
Each of the $N$ individuals whishes  to have a velocity $\UU_i$ which depends on its position only: $\UU_i = \UU(\qq_i)$, where $\UU(\cdot)$ is some given velocity field over $\setR^2$ (typically $\UU = -\nabla D$, where $D$ is the  geodesic distance to the exit).  
To account for  non-overlapping, it is assumed that the actual velocity $\uu=(\uu_1,\dots,\uu_N)$ is the $\ell^2$-projection  of $\tilde \UU=(\UU_1,\dots,\UU_N) =(\UU(\qq_1),\dots,\UU(\qq_N)) $ onto the cone of feasible velocities  $C_\qq$ (i.e. the set of velocities which do not lead to a violation of the non-overlapping constraint). The model takes the form
\begin{equation}
\left \{
   \begin{array} {rcl} 
\displaystyle  \frac {d \qq}{dt} &=& \uu \vseq\\
      \uu  &=& P_{C_\qq}\tilde \UU. \\
   \end{array}
   \right . 
\label{eqmicro}
\end{equation}



In the spirit of this microscopic approach, the model we propose here  rests on the two following principles
\begin{enumerate}
\item  the pedestrian population is described by a density  $\rho$ which is subject to remain below a certain maximal value (equal to $1$ in what follows), this density follows an advection equation, 

\item  the advecting field $\uu$ is the closest, among admissible fields (i.e. which do not lead to a violation of the constraint), to some spontaneous  field $\UU$, which corresponds to the strategy people would follow in the absence of others.
\end{enumerate}

If we denote by $C_\rho$ the cone of admissible velocities (i.e. set of velocities which do not increase density in already saturated zones, see next section for a proper definition),
the model takes the following form
\begin{equation}
\left \{
   \begin{array} {rcl} 
\displaystyle\partial _t \rho  + \nabla \cdot \rho \uu &=& 0\vseq \\
      \uu  &=& P_{C_\rho}\UU, \\
   \end{array}
   \right . 
\label{eqinformelle}
\end{equation}
where the projection is meant in the $L^2$ sense.
As a matter of fact, in the same way as Cauchy-Lipschitz theory for ODE's no longer applies for CD in the microscopic case,  we cannot use classical methods to study Equation~\eqref{eqinformelle}, as well as most of the PDE's we could encounter in the CD macroscopic models. This is due in particular to the lack of regularity of the velocity $\mathbf{u}$ (whose natural regularity is $L^2$), which prevents us to apply the characteristic method or even DiPerna-Lions theory. The non-continuous dependence of the operator $P_{C_{\rho}}$ with respect to $\rho$ is another source of problems.

Instead, we will see that this PDE corresponds to a Gradient Flow in the Wasserstein space (i.e. the space of probability measures endowed with the distance $W_{2}$ induced by the optimal transport under quadratric cost), provided that the spontaneous velocity field has a gradient structure: $\mathbf{U} = - \nabla D$. This means that we consider the functional 
$$\Phi(\rho) = \begin{cases}
\displaystyle \int_{\Omega} D(x) \rho(x) dx & \textmd{if} \: \rho \leq 1\\
+ \infty & \textmd{otherwise}
\end{cases}$$
and we look for the curve of measures $\rho(t,.)$ which follows the steepest descent direction of $\Phi$ starting from a given datum $\rho^0$. This curve will happen to solve equation \eqref{eqinformelle}.
This is a general and very efficient method to find solutions to certain evolution PDE's which been made possible by the theory of optimal transportation. This theory owes its origin to Kantorovich\cite{Kan}, but has been widely developped thereafter (see the books by Villani\cite{Vil1,Vil2}).
Several equations have been approached by this method, for instance the classical heat equation, as well as the Fokker-Planck or the porous media equations (see Refs.~\cite{Ott1,Ott2,CarGuaTos}). Notice that, as the functional which is used to produce the porous media equation as a Gradient Flow is 
$$
\rho \mapsto \displaystyle \int \rho(x)^m dx,
$$
our case can be considered as its formal limit when $m$ tends to infinity.
All the theory of Gradient Flow in Wasserstein Spaces is treated in the reference book by Ambrosio-Gigli-Savar\'e\cite{Amb} and one of the key assumptions is the $\lambda$--convexity of the functional, which ensures better estimates. On the other hand, some existence results can be obtained without this assumption, but they have to be treated carefully by hand, as it happens in Ref. \cite{BlaCalCar}. In our case, even if we suppose $D$ to be $\lambda$--convex, we face the same kind of difficulties if we want to add the presence of an exit door on the boundary of $\Omega$ where the measure can concentrate (see Section~\ref{sec:eulerian}).

The paper is organized as follows:
In Section~\ref{sec:eulerian} we present the model in the Eulerian setting and a related discrete minimizing movement scheme (MMS).
We explain how a straightforward use of a convergence  theorem in \cite{Amb} asserts a convergence of the trajectories for the discrete MMS to some continuous pathline.  Identification of this limit with a solution to the initial problem can be done unformally. Yet some technical obstacles (in particular the handling of walls) prevent us from obtaining a fully rigorous proof based on this approach. 
The actual proof  of  convergence to a solution of the crowd motion model is based on alternative arguments. The end of this section describes this convergence results. 
As the presence of an exit raises some very  specific technical difficulties, we propose in 
  Section~\ref{sec:exnoex} a proof in the case there is no exit.
  The proof in the general case in given in  Section~\ref{sec:exexit}.
To illustrate the convergence theorem, we present in Section 5 an idealized (yet non trivial) situation where both eulerian solutions and discrete MMS trajectories can be described with accuracy.
 Finally, we discuss in Section 6 the limitations of this model and its possible extensions to other fields of natural sciences. In particular, we explain why we developped the theory in any dimension although dimensions greater than two do not make clear sense as far as crowd motion is concerned.

\section{The eulerian  model and its gradient flow formulation}
\label{sec:eulerian}

\subsection{Eulerian model}
The model we propose is designed to handle emergency evacuation situations~: the behaviour of individuals is based on optimizing their very own trajectory, regardless of others, but the fulfillment of individual strategies is made impossible because of congestion.

The model takes the following form: given a domain $\OO$ (the building), whose boundary $\GG$   is composed of $\GG_{out}$ (the exit) and  $\GG_w$ (the walls), we describe the current distribution of people by a measure $\rho$ of given mass (say $1$ without loss of generality) supported within $\overline \OO$. 
To model the fact that people getting through the door are out of danger, yet keeping a constant total mass without having to model the exterior of the building, we shall  assume that $\rho$ may concentrate on $\GG_{out}$.
In this spirit, we denote by $K$ the set of all those probablity measures over $\setR^2$ that are supported in $\overline \OO$, and  that  are the sum of a diffuse part, with density between $0$ and $1$, in $\OO$, and a singular part carried by $\GG_{out}$.

\begin{figure}[!h]
\begin{center}
\psfrag{OO}[l]{$\OO$}
\psfrag{GGout}[l]{$\GG_{out}$}
\psfrag{GGw}[l]{$\GG_{w}$}
\includegraphics[width=0.6\textwidth]{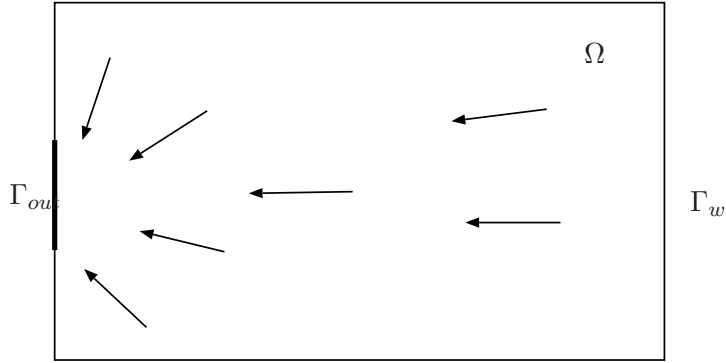}
\caption{Geometry.}
\label{geom}
\end{center}
\end{figure}

We shall denote by $\UU$ the spontaneous velocity field: $\UU(x)$  represents the velocity  that  an individual  at $x$ would have if he were alone. It is taken equal to $0$ outside $\OO$.
	The set $C_\rho$ of feasible  velocities corresponds to all those fields  which do not increase $\rho$ on the saturated zone (unformally,  $\nabla\cdot \uu \geq 0$ 	in $[\rho=1]$), and which account for walls (people do not walk through them). 
As we plan to define $C_\rho$ as a closed convex set in $L^2(\OO)$, those constraints  do not make sense as they are, and we shall favor a dual definition of this set. Let us introduce the ``pressure'' space 
 $$
 H^1_\rho = \{ q\in H^1(\OO)\, ,\,\, q \geq 0 \hbox { a.e. in  }\OO\, ,\,\, q(x) = 0 \hbox { a.e.  on } [\rho<1]\virg
 q_{|\GG_{out}} = 0\}. 
 $$
%
The proper definition of $C_\rho $ reads
\begin{equation}
\label{eq:defCrho}
C_\rho = \{ \vv \in L^2(\OO)^2\virg \int_\OO \vv \cdot \nabla q \leq 0\quad \forall q \in   H^1_\rho
\}.
\end{equation}

The model is based on the assumption that the actual instantaneous velocity field is the feasible field which is the closest to $\UU$ in the least-square sense, i.e. it is defined as the $L^2$-projection of $\UU$ onto the closed convex cone $C_\rho$. 
Finally the problem consists in finding a trajectory $t\mapsto \rho(t) \in K$  which is advected by $\uu$, i.e. such that $(\rho,\uu)$ is a (weak) solution of the transport equation in $\setR^2$
\begin{equation}
\label{eq:transp}
\partial_t \rho + \nabla \cdot(\rho \uu) =0,
\end{equation}
where $\uu$ verifies, for almost every $t$, 
\begin{equation}
\label{eq:proj}
\uu =P_{C_\rho}\UU.
\end{equation}

\begin{remark}
\label{remark:exit}
The fact that  $\GG_{out}$ is likely to carry some mass calls for some proper definition  of the velocity on this zero-measure set.  As the exit plays the role of a reservoir in our model, 
we shall actually consider that all feasible  fields vanish on $\GG_{out}$, so that velocity $\uu$ given  by~(\ref{eq:proj}) will be considered as defined  Lebesgue-a.e. in $\OO$ and vanishing on $\GG_{out}$.

\end{remark}

\begin{remark}
 {\em Boundary conditions (walls and exit)}. \\
The unilateral divergence constraint and the behaviour at walls and exit are implicitly contained in the dual expression of $C_\rho$, as illustrated by the following considerations. We assume in this remark that $[\rho=1]=\overline\omega$ where $\omega\subset \Omega$ is a smooth subdomain,
and that all 
fields are smooth.
 First of all, by taking tests pressures which are smooth and compactly supported in $\omega$, we obtain $\nabla\cdot \uu \geq 0$ in the saturated zone.
As the pressure vanishes on $\GG_{out}$, the velocity is free on that part of the boundary (free outlet condition, as in Darcy flows). 
Consider now  a situation  where  the saturated zone covers the wall $\GG_{w}$. For any smooth function $\varphi$  defined on $\GG_w$ consider a sequence of extensions $\varphi_\eps$ supported within $\omega \cup \GG_w$, which converges to $0$ in $ L^2(\OO)$. Then 
$$
 \int_\OO \uu \cdot \nabla \varphi_\eps \leq 0 \quad \forall\eps >0
 $$
 implies
 $$
- \int_\OO \varphi_\eps \nabla \cdot \uu  + \int_{\GG_w} \varphi_\eps \uu\cdot \nn \leq 0   \quad \forall \eps >0.
$$
As the first term goes to $0$ with $\eps$ we obtain that the velocity necessarily enters the domain on the saturated wall (what we adressed before as ``people do not walk through walls'').
\end{remark}

\subsection{Gradient flow formulation}

In this section we introduce a discrete evolution problem in the Wasserstein space, whose limit will be the gradient flow of a suitable functional,  and we establish unformally the link between this new problem and the crowd motion model. The formal equivalence, which will be proved rigorously in the following sections, will be satisfied in the case where $\UU=-\nabla D$ is the opposite of a gradient.

Let us denote by   $ \mathcal{P}_{2}$ the set of probablity measures over $\mathbb{R}^2$ endowed with the Wasserstein distance, and by
\begin{equation}
\label{eq:defK}
K = \{ \rho \in  \mathcal{P}_{2}\, , \,\, \hbox{supp}( \rho) \subset \overline{\OO}\virg \rho = \rho_{out} + \rho_{\OO}\virg 
\rho_{\OO}(x) \leq 1\,\, \hbox{ a.e.} , \,\, \hbox{supp}( \rho_{out}) \subset\GG_{out}
\}
%
\end{equation}
the set of feasible densities.
Let an initial density $\rho^0$ be given, and $\tau > 0$ a time step. We build $\rho^0_\tau = \rho^0$, $\rho^1_\tau$, \dots 
 as follows

\begin{equation}
\rho_\tau^{k}  = \mathop{\textmd{argmin}}\limits_{\mathcal{P}_{2}(\mathbb{R}^d)} \left\{ J(\rho) + \mathbf{I}_{K} (\rho)+ \dfrac{1}{2 \tau} W_{2}^2 (\rho, \rho_\tau^{k-1}) \right\},
\end{equation}
where $W_{2}$ is the Wasserstein distance, $J$ is the dissatisfaction functional defined as
\begin{equation}\label{defiJ}
J(\rho):= \displaystyle \int_{\Omega} D(x) \rho(x) \, dx,
\end{equation}
and $\mathbf{I}_K$ is the indicatrix of $K$~:
$$
\mathbf{I}_K(\rho) = \left | 
\begin{array}{ccl}
0 & \hbox{if} & \rho \in K \\
+\infty & \hbox{if} & \rho \notin K .
\end{array}
\right . 
$$
The function $D$ is typically the distance to the door $\Gamma_{out}$, and to $D$ we associate a vector field $\UU=-\nabla D$. It is important in order to have vanishing velocities on the door that $D$ is minimal and constant on $\Gamma_{out}$.

We admit here that under reasonable assumptions this process is indeed an algorithm (i.e. $\rho_\tau^{k+1}$ is uniquely defined as the minimizer of the function  above), and we denote by $\rho_\tau$ the piecewise constant interpolate of $\rho^0_\tau$,  $\rho^1_\tau$, \dots.

As $\tau$ goes to $0$, by Prop. 2.2.3, Th. 2.3.1, and Th. 11.1.3  in \cite{Amb}, $\rho_\tau$ converges to some trajectory $t\mapsto \rho$ in $K$, which is a (weak) solution to 
$$
\partial_t \rho  + \nabla \cdot \left ( \rho \uu\right) = 0,
$$
where $\uu$ is such that, for almost every $t$,
$$
\uu \in - \partial \left ( J + I_K \right ) (\rho),
$$
where $\partial \Psi$ denotes the strong subdifferential of   $\Psi$.
Furthermore $\uu$ minimizes the $L^2$ norm among all those fields in the subdifferential above.

Let us now prove unformally that this characterizes the instantaneous velocity as the projection of $\UU = -\nabla D$ onto $C_\rho$.
The subdifferential  of a function $\Psi$ at $\rho$ in the Wasserstein setting is defined as  the set of fields $\uu$ such that 
$$
\Psi(\rho) + \int_\OO \left <\uu,\bft(x)-x \right >d\rho(x) \leq \Psi (\bft _{\#} \rho)
+ 
o(|| {\bft- \bfi} ||),
$$
where $\bft$ denotes a transport map acting on  $\rho$. Note that the previous inequality does not provide any information as soon as $\bft _{\#} \rho$ is not feasible (in that case the right-hand side is $+\infty$).
Let us  consider a feasible field $\vv \in C_\rho$, and let us assume that,  for $\eps$ small enough, $\bft_\eps = \bfi + \eps \vv$  pushes forward $\rho$ onto a measure in $K$
(this is not true in general, see Remark~\ref{remark:feas}). Note that $\bft_\eps$  is defined $\rho$-almost everywhere, with $\GG_{out}$ carrying some mass, but as it  vanishes on  $\GG_{out}$ (see Remark~\ref{remark:exit}),  the singular part of $\rho$  remains unchanged.
Having $\eps$ go to $0$ in the subdifferential inequality, we obtain 
$$
\int \nabla D\cdot \vv\, d\rho(x)  + \int \uu \cdot \vv \, d\rho(x) \leq 0,
$$
so that $\uu + \nabla D = \uu -\UU$ belongs to $C_\rho^\circ$, the polar cone to $C_\rho$.
As $\uu$ minimizes the $L^2$ norm over $\UU+C_\rho$, $\uu$ identifies with the projection of $\UU$ onto $C_\rho$, which ends this unformal proof.

\begin{remark}
\label{remark:feas}
In general, there exist feasible densities $\rho\in K$ (defined by~(\ref{eq:defK})) and fields $\vv\in C_\rho$ (defined by~(\ref{eq:defCrho})) such that $(\bfi+\eps\vv)_{\#} \rho$ exits $K$ for any $\eps >0$, this is why the considerations above  do not make a rigorous proof. Consider for example $\omega $ a dense open subset in $\OO$,  with a small  measure, and define $\rho$ as $\ONE_{\omega^c}$. 
The pressure space is $\{ 0 \}$, and $C_\rho$ is $L^2(\OO)$: any field is feasible. If one considers now a strictly contractant field (with negative divergence), it is clear that $(\bfi+\eps\vv)_{\#} \rho \notin K$ for any $\eps>0$. Notice also that this kind of paradox does not depend on the fact that we chose a ``linear'' perturbation $(\bfi+\eps\vv)$, since the same would happen if one, instead, perturbs the identity by following the flow of the vector field $\vv$ for a time $\ve$ (which is classically a better choice in order to satisfy the density constraint).
\end{remark}

As explained  in the previous remark, the approach carried out in this section  is not a rigorous proof that the advecting field is actually the projection of $\UU$ onto $C_\rho$. 
We conjecture that projecting $(\bfi+\eps\vv) _{\#}  \, \rho $ onto $K$ (for the Wasserstein distance) introduces a perturbation which is negligible compared to $\eps$, so that $\vv$ may actually be used as a test-function,  but this conjecture raises some technical issues which we were not able to solve. 
In what follows we give an alternate proof which circumvents the necessity to characterize $\partial(J + I_K)$.

\subsection{Notations and statement of the main result}

We first recall some results on the continuity equation: let $(\rho(t,.))_{t>0}$ be a family of density measures on $\mathbb{R}^d$, and $\mathbf{v}: (t,x) \in \mathbb{R}^+ \times \mathbb{R}^d \mapsto \mathbf{v}(t,x) \in \mathbb{R}^d$ be a Borel velocity field such that 
\begin{equation}
\displaystyle \int_{0}^T \int_{\mathbb{R}^d} |\mathbf{v}(t,x)| \rho(t,x) \, dx \, < \, + \infty.
\end{equation}
We say that $(\rho, \mathbf{v})$ satisfies the continuity equation with initial condition $\rho^0$
\begin{equation}
\left\{ \begin{array}{rcl}
\partial_{t} \rho + \nabla \cdot (\rho \mathbf{v} ) & = & 0\vseq \\
\rho(0,.) & = & \rho^0
\end{array} \right.
\label{continuity_equation}
\end{equation}
if for all $\varphi \in \mathcal{C}_{c}^{\infty}([0,T[ \times \mathbb{R}^d)$ we have
\begin{equation}
\int_{0}^T \int_{\mathbb{R}^d} (\partial_{t} \varphi (t,x) + \nabla \varphi (t,x) \cdot \mathbf{v}(t,x) ) \, \rho(t,x) \, dx + \int_{\mathbb{R}^d} \varphi(0,x) \rho^0(x) \, dx = 0.
\end{equation}
Let us recall that if $\rho(t,.)$ is a solution of the continuity equation, there exists a narrowly continuous curve $\tilde \rho(t,.)$ such that $\rho(t,.) = \tilde \rho(t,.)$ for a.e. t. In general, we will always focus on this continuous representation.\\

We now detail the construction of a discrete family of densities $(\rho_\tau^k)$ that approches in a sense we will precise later the solution of the continuity equation we are interested with. For a fixed time step $\tau >0$, we define the sequence $(\rho_\tau^k)$ of density measures on $\Omega$ using the recursive scheme:
\begin{equation}
\left\{ \begin{array}{rcl}
\rho_\tau^0 & = & \rho^0\\
\rho_\tau^{k} & \in & \mathop{\textmd{argmin}}\limits_{\mathcal{P}_{2}(\mathbb{R}^d)} \left\{ J(\rho) + \mathbf{I}_{K} (\rho)+ \dfrac{1}{2 \tau} W_{2}^2 (\rho, \rho_\tau^{k-1}) \right\},
\end{array} \right.
\label{min_scheme}
\end{equation}
where $W_{2}$ is the Wasserstein distance, and $J$ is the dissatisfaction functional defined in \eqref{defiJ}. 

This construction is a minimizing movement scheme as described by DeGiorgi and Ambrosio in \cite{DeG,MovMin} and then - in the framework of probability measures - in \cite{Amb,AmbSav} with functional $\Phi(\rho) = J(\rho) + \mathbf{I}_{K}(\rho) $.


We define on $\mathring{\Omega}$ the discrete velocities: $\mathbf{v}_\tau^k = \dfrac{\mathbf{i} - \mathbf{t}_\tau^k}{\tau}$, where $\mathbf{t}_\tau^k$ is the unique optimal transport function from $\rho_\tau^{k}$ to $\rho_\tau^{k-1}$, which is well defined on $\mathring{\Omega}$ (but not necessarely on $\Gamma_{out}$, due to the singular part of $\rho_\tau^{k}$). We also define $\mathbf{E_\tau^k} = \rho_\tau^k \mathbf{v}_\tau^k$ on $\mathring{\Omega}$
(by abuse of notation, we will write $\mathring{\Omega}$ instead of $\Omega$ when we want to stress that we are not considering the boundary).
We can interpolate these discrete values $(\rho_\tau^k, \mathbf{v}_\tau^k, \mathbf{E_\tau^k})_{k \geq 0}$ by the piecewise constant functions defined by:
\begin{equation}
\left\{ \begin{array}{l}
\rho_\tau (t,.) = \rho_\tau^k\\
\mathbf{v_{\tau}} (t,.) = \mathbf{v}_\tau^k\\
\mathbf{E_{\tau}} (t,.) = \mathbf{E_\tau^k}
\end{array} \right. \quad \textmd{if} \; t \in \;  ](k-1) \tau, k \tau].
\end{equation}

\noindent Our goal is to prove that $\rho_\tau$ converges when $\tau \rightarrow 0$ to a solution of the continuity equation (\ref{continuity_equation}). Here is our main result:\\

\begin{theorem}
Let $\Omega$ be a convex bounded set of $\mathbb{R}^d$, $D: \mathbb{R}^d \rightarrow \mathbb{R}$ a continuous $\lambda$-convex function, $\rho^0$ a probability density, and $(\rho_\tau^k)$ constructed following the recursive scheme (\ref{min_scheme}).\\
Then there exists a family of probability densities $(\rho(t,.))_{t>0}$, and a family of velocities $(\mathbf{u}(t,.))_{t>0}$ such that $(\rho_\tau(t,.), \mathbf{E_{\tau}}(t,.))$ narrowly converges to $(\rho(t,.), \rho(t,.) \, \mathbf{u}(t,.))$ for a.e. $t$. Moreover, $(\rho, \mathbf{u})$ satisfies the continuity equation:
\begin{equation}
\left\{ \begin{array}{rcl}
\partial_{t} \rho + \nabla . (\rho \mathbf{u}) & = & 0\\
\mathbf{u}(t,.) & = & P_{\mathcal{C}_{\rho(t,.)}} \mathbf{U} \quad \textmd{for a.e.} \; t\\
\rho(0,.) & = & \rho^0
\end{array} \right.
\end{equation}
where $\mathbf{U} = - \nabla D$, and $\mathcal{C}_{\rho(t,.)}$ is defined in~(\ref{eq:defCrho}).
\label{theorem}
\end{theorem}

\noindent We will first prove this theorem in the particular case where there is no exit. In the following section, we thus assume that $\Gamma_{out} = \emptyset$, which will imply that all the measures are absolutely continuous with respect to the Lebesgue measure. Then we will extend the proof to the general case.

\begin{remark}
We chose to assume here a $\lambda-$convexity hypothesis on $D$ both in order to clarify some statements, which are easier to state and prove under this assumption (see for instance Lemma \ref{avec uniq} and the subsequent Remark \ref{sans uniq}) and because the typical case we think of is $D=d(\cdot,\Gamma_{out})$, where $\Gamma_{out}$ is a flat part of the boundary of the convex set $\Omega$. This implies that $D$ is convex as well. It would be interesting to study the case of non-convex domains $\Omega$ (for instance with obstacles), and use the geodesic distance for computing $D$, which would lead to a non-$\lambda-$convex function, but this is not yet possible by means of our techniques, since one should work with the Wasserstein distance $W_2$ computed w.r.t. the geodesic distance itself, which is not much studied.

Anyway, it can be checked that the only point throughout the paper where $\lambda-$convexity is used is the proof of Lemma \ref{avec uniq}, but Remark \ref{sans uniq} explains how to get rid of this assumption: this means that, for existence purposes, this assumption may be withdrawn.
On the other hand, the $\lambda-$convexity assumption is typical in this gradient flow problems, because it allows for uniqueness and stability results, and we think that similar results could be achieved in our case as well.
\end{remark}

\section{Existence result in a domain with no exit}
\label{sec:exnoex}

\subsection{Technical lemmas}

Since we will make a strong use of optimality conditions in terms of the dual problem in Monge-Kantorovitch theory, let us briefly recall what we need. 

Given the two probabilities $\mu$ and $\nu$ on $\Ov$ we always have
$$\frac 12 W_2^2(\mu,\nu)=\max\left\{\int_{\Omega}\!\varphi \,d\mu+\!\!\int_{\Omega}\!\psi \,d\nu,\,\phi, \psi \in C^0(\Ov) :\,\phi(x)+\psi(y)\leq\frac 12 |x-y|^2\right\},
$$
the maximum being always realized by a pair of $c-$concave conjugate functions $(\varphi,\psi)$ with  $\varphi=\psi^c$ and $\psi=\varphi^c$, where the $c-$transform of a function $\chi$ is defined through
$$\chi^c(y)=\inf_{x\in\Omega}\frac 12 |x-y|^2-\chi(x)$$
(with generalizations to other costs $c$ rather than the square of the distance). We will call Kantorovitch potential from $\mu$ to $\nu$ (resp., from $\nu$ to $\mu$) any $c-$concave function $\varphi$ (resp., $\psi$) such that $(\varphi,\varphi^c)$ (resp., $(\psi^c,\psi)$) realizes such a maximum. We have uniqueness of the optimal pair as soon as one of the support of one of the two measures is the whole domain $\Ov$.

\vspace*{0.2cm}
\begin{lemma}\label{avec uniq} Let $D: \mathbb{R}^d \mapsto \mathbb{R}$ $\lambda$--convex, and $\bar \rho \in K$. Then, there exists $\tau^*$ such that for all $\tau < \tau^*$:\\
(i) The functional $\phi(\rho) = \displaystyle \Phi(\rho) + \dfrac{1}{2 \tau} W_{2}^2(\rho, \bar \rho)$ admits a unique minimizer $\rho_{m}$.\\
(ii) There exists a Kantorovitch potential  $\bar \varphi$ from $\rho_{m}$ to $\bar \rho$, such that:
\begin{equation}
\displaystyle \int_{\Omega} \left( D + \dfrac{\bar \varphi}{\tau} \right) \rho \; \geq \;  \int_{\Omega} \left( D + \dfrac{\bar \varphi}{\tau} \right) \rho_{m} \quad \textmd{for all} \; \rho \leq 1 \; \textmd{a.e.} .
\label{ineq_kant_form}
\end{equation}
\label{ineq_kant}
\end{lemma}

\begin{proof} (i) The existence of a minimizer can easily be proved using a minimizing sequence of $\phi(\rho)$. Let $\rho_{1}, \rho_{2}$ be two different minimizers, and $\mathbf{r_{i}}$ the optimal transport between $\bar \rho$ and $\rho_{i}$. We define $\mathbf{r_{t}}:= (1-t) \mathbf{r_{1}} + t \, \mathbf{r_{2}}$ and $\rho_{t}:= \mathbf{r_{t}}_{\#} \rho_{1}$, for $t \in \; ]0,1[$. 
We know that $\rho_{t} = \dfrac{\rho}{|\det \nabla \mathbf{r_{t}}|} \circ (\mathbf{r_{t}})^{-1}$.  As $M \mapsto (\det M)^{-1}$ is convex on the set of positive definite matrices $S_{d}^{++}$, and $\nabla \mathbf{r_{i}} \in S_{d}^{++}$, we have~: 
$$\rho_{t}(x) \leq \left( \dfrac{1-t}{\det \nabla \mathbf{r_{1}}} + \dfrac{t}{\det \nabla \mathbf{r_{2}}} \right) \bar \rho \circ (\mathbf{r_{t}})^{-1}(x).$$
 We also know that $\rho_{1}$ and $\rho_{2}$ are admissible, therefore: $\dfrac{\bar \rho}{\det \nabla \mathbf{r_{i}}} \leq 1$ a.e., and we obtain: $\rho_{t} \leq 1$. We have then
  $$\phi(\rho_{t}) = \displaystyle \int_{\Omega} D((1-t) \mathbf{r_{1}}(x) + t \, \mathbf{r_{2}}(x)) \bar \rho (x) \, dx + \dfrac{1}{2 \tau} W_{2}^2(\rho_{t}, \bar \rho).$$
   Since $D$ is $\lambda$--convex
 $$D((1-t) \mathbf{r_{1}}(x) + t \, \mathbf{r_{2}}(x)) \leq (1-t) D(\mathbf{r_{1}}(x)) + t D(\mathbf{r_{2}}(x)) - \dfrac{\lambda}{2} t (1-t) |\mathbf{r_{1}}(x) - \mathbf{r_{2}}(x)|^2.$$
Moreover, $W_{2}^2(., \bar \rho)$ is $1$--convex along the interpolation $\rho_{t}$ (see lemma 9.2.1 p. 206 in Ref \cite{Amb}), therefore, for $\tau$ small enough, we have
 $$
 \phi(\rho_{t}) < (1-t) \phi(\rho_{1}) + t \phi(\rho_{2}) = \inf_{K} \phi(\rho),
 $$
  which is absurd.

\noindent (ii) We first assume that $\bar \rho > 0$ a.e., which implies that the Kantorovich potential $\bar{\varphi}$ from $\rho_{m}$ to $\bar \rho$, satisfying $\bar{\varphi}(x_{0}) = 0$ (with $x_{0}$ any fixed point in $\Omega$), is unique. Let us define a small perturbation of $\rho_{m}$: let $\rho \leq 1$ be a probability density, $\varepsilon > 0$ and $\rho_{\varepsilon}: = \rho_{m} + \varepsilon (\rho - \rho_{m} )$. As $\rho_{m}$ minimizes $\phi(\rho)$, we have:
\begin{equation}
J(\rho_{\varepsilon}) - J(\rho_{m}) + \dfrac{1}{2 \tau} ( W_{2}^2( \rho_{\varepsilon}, \bar \rho) - W_{2}^2 (\rho_{m}, \bar \rho) ) \; \geq \; 0.
\label{inegdisc}
\end{equation}

\noindent The first part of the left side of the inequality can easily be calculated: 
$$J(\rho_{\varepsilon}) - J(\rho_{m}) = \displaystyle \int_{\Omega} D(x) (\rho_{\varepsilon} - \rho_{m}) (x) \, dx = \varepsilon \int_{\Omega} D(x) (\rho - \rho_{m}) (x) \, dx . $$
Let us estimate the second part: we denote by $(\varphi_{\varepsilon}, \psi_{\varepsilon})$ some Kantorovich potentials associated to $\bar \rho$ and $\rho_{\varepsilon}$. We have
$$\left\{ \begin{array}{lll}
\dfrac{1}{2} W_{2}^2 (\rho_{\varepsilon}, \bar \rho)  & = & \displaystyle \int_{\Omega} \varphi_{\varepsilon}(x) \rho_{\varepsilon}(x) \, dx + \int_{\Omega} \psi_{\varepsilon}(y) \bar \rho (y) \, dy\\
\dfrac{1}{2} W_{2}^2 (\rho_{m}, \bar \rho)  & \geq & \displaystyle \int_{\Omega} \varphi_{\varepsilon}(x) \rho_{m}(x) \, dx + \int_{\Omega} \psi_{\varepsilon}(y) \bar \rho (y) \, dy,
\end{array} \right.$$
where $\phi_\ve$ is a Kantorovitch potential from $\rho_\ve$ to $\bar\rho$. Thus:
$$\dfrac{1}{2} ( W_{2}^2( \rho_{\varepsilon}, \bar \rho) - W_{2}^2 (\rho_{m}, \bar \rho) ) \leq \displaystyle \int_{\Omega} \varphi_{\varepsilon}(x) (\rho_{\varepsilon} - \rho_{m}) (x) \, dx = \varepsilon \int_{\Omega} \varphi_{\varepsilon}(x) (\rho - \rho_{m}) (x) \, dx ,$$
and we can deduce from inequality (\ref{inegdisc}) that:
$$ \displaystyle \int_{\Omega} D(x) (\rho - \rho_{m}) (x) \, dx + \dfrac{1}{\tau}  \int_{\Omega} \varphi_{\varepsilon}(x) (\rho - \rho_{m}) (x) \, dx \geq 0 \quad \textmd{for all admissible} \, \rho .$$
Let $\varepsilon$ tend to $0$: $\varphi_{\varepsilon}$ converges to the unique Kantorovich potential $\bar{\varphi}$ from $\rho_{m}$ to $\bar \rho$. This gives
$$ \displaystyle \int_{\Omega} D(x) (\rho - \rho_{m}) (x) \, dx + \dfrac{1}{\tau}  \int_{\Omega} \psi^c(x) (\rho - \rho_{m}) (x) \, dx \; \geq \; 0 \quad \textmd{for all admissible} \, \rho.$$

We now prove the general case: let $\bar \rho_{\delta} > 0$ a.e., $\bar \rho_{\delta} \leq 1$ a.e., such that $\bar \rho_{\delta}$ converges to $\bar \rho$ when $\delta$ tends to $0$. Using (i), there exists a unique minimizer $\rho_{m, \delta}$ of $\phi_{\delta}(\rho):= \displaystyle \int_{\Omega} D \rho + I_{K} + \dfrac{1}{2 \tau} W_{2}^2(\rho, \bar \rho_{\delta})$, and it converges to $\rho_{m}$ as $\delta$ tends to $0$. Moreover, we have proved that: $$ \displaystyle \int_{\Omega} D(x) (\rho - \rho_{m, \delta}) (x) \, dx + \dfrac{1}{\tau}  \int_{\Omega} \bar \varphi_{\delta}(x) (\rho - \rho_{m, \delta}) (x) \, dx \; \geq \; 0 \quad \textmd{for all admissible} \, \rho,$$
with $\bar \varphi_{\delta}$ that converges to a Kantorovitch potential $\bar \varphi$. Taking the limit $\delta \rightarrow 0$, we obtain the desired inequality.  For this kind of arguments concerning optimality for transport costs and other functionals, see for instance Ref. \cite{ButSan}.
\end{proof}

\begin{remark}\label{sans uniq} if $D$ is not $\lambda$--convex, we cannot prove uniqueness of the minimizer of $\phi$. However, if $\rho_{m}$ is a minimizer, it still satisfies inequality (ii). Indeed, in the second part of the proof of (ii), we can define $\rho_{m, \delta}$ as a minimizer of $\phi_{\delta}(\rho) + c_{\delta} W_{2}^2(\rho, \rho_{m})$, where $c_{\delta}\to 0$ (so that the optimality condition that we see at the limit $\delta\to 0$ disregards this term), but slowly (so that it makes $\rho_{m, \delta}$ converge to $\rho_{m}$). Obviously this kind of argument was not necessary if one only wanted to prove this optimality condition for one minimizer $\rho_m$, and not for every minimizer.
\label{selection}
\end{remark}

\vspace*{0.2cm}
\begin{lemma} (Decomposition of the spontaneous velocity):\\
The spontaneous velocity $\mathbf{U} = - \nabla D$ can be written as follows:
\begin{equation}
\mathbf{U} = \mathbf{v}_\tau^k + \nabla p_{\tau}^k \quad \textmd{with} \; p_{\tau}^k \in H^1_{\rho_{\tau}^k}.
\end{equation}
\label{decomposition}
\end{lemma}

\begin{proof} Using the previous lemma, we know that there exists a Kantorovich potential $\bar\varphi$ from $\rho_\tau^{k}$ to $\rho_\tau^{k-1}$ such that $\rho_\tau^{k}$ is a solution of the minimizing problem:
$$ \rho_\tau^{k} \, \in \, \mathop{\textmd{argmin}}\limits_{\rho \in K} \left\{ \displaystyle \int_{\Omega} D(x) \rho (x) \, dx + \dfrac{1}{\tau}  \int_{\Omega} \bar\varphi(x) \rho (x) \, dx \right\},$$
which imposes:
$$\left\{ \begin{array}{ccl}
\rho_\tau^{k} = 1 & \textmd{on} & [F < l]\\
\rho_\tau^{k} \leq 1 & \textmd{on} & [F = l]\\
\rho_\tau^{k} = 0 &  \textmd{on} & [F > l],
\end{array} \right.
\;\;\hbox{ with } \;\;F\; :\; \left | \begin{array}{rcl}
\Omega & \rightarrow & \mathbb{R}\\
x & \mapsto & D(x) + \dfrac{\bar\varphi(x)}{\tau} ,
\end{array} \right.
$$ 
and $l \in \mathbb{R}$ chosen such that $\rho_\tau^{k}$ satisfies: $\displaystyle \int_{\Omega} \rho_\tau^{k} \, dx = 1$.\\
We can then define a pressure like function:
\begin{equation}
p_{\tau}^k(x):= (l - F(x))_{+} = \left( l - D(x) - \dfrac{\bar\varphi(x)}{\tau} \right)_{+}
\end{equation}
which satisfies: $p_{\tau}^k \geq 0 $, and $p_{\tau}^k = 0$ on $[\rho_\tau^{k} < 1 ]$, therefore $p_{\tau}^k \in H^1_{\rho_{\tau}^k}$.\\
Moreover, on $[\rho^k_\tau>0]$, we have: $\nabla p_{\tau}^k = - \nabla D - \dfrac{\nabla \bar\varphi}{\tau}$ (where the density vanishes $\mathbf{v}^k_\tau$ may be modified at will, so that we can keep the same formula). Since we have 
$$
\mathbf{v}_\tau^k = \dfrac{\mathbf{i} - \mathbf{t}_\tau^k}{\tau} = \dfrac{\nabla \bar\varphi}{\tau},
$$
we  get the desired decomposition for the spontaneous velocity  : $\mathbf{U} = \mathbf{v}_{\tau}^k + \nabla p_{\tau}^k$. \end{proof}

Let us now define the densities $\tilde \rho_\tau(t)$ that interpolate the discrete values $(\rho_\tau^k)$ along geodesics:
\begin{equation}
\tilde \rho_\tau(t) = \left( \dfrac{t - (k-1) \tau}{\tau} (\mathbf{id} - \mathbf{t}_\tau^k) + \mathbf{t}_\tau^k \right)_{\#} \rho_\tau^{k} .
\end{equation}
We also define $\mathbf{\tilde v_{\tau}}(t,.)$ as the unique velocity field such that $\mathbf{\tilde v_{\tau}}(t,.) \in \textmd{Tan}_{\tilde \rho_{t}} \, \mathcal{P}_{2}(\mathbb{R}^d)$ and $(\tilde \rho_\tau, \mathbf{\tilde v_{\tau}})$ satisfy the continuity equation. As before, we define: $\mathbf{\tilde E_{\tau}}= \tilde \rho_\tau \mathbf{\tilde v_{\tau}}$.

After these definitions we will give some a priori bounds on the curves, the pressures and the velocities that we defined. In order to get these bounds, we need to start from some estimates which are standard in the framework of Minimizing Movements.
The sequence $(\rho_\tau^k)_k$ satisfies an estimate on its variation which  gives a H\"older and $H^1$ behavior. From the minimality of $\rho_\tau^{k}$, compared to $\rho_\tau^{k-1}$, one gets
 $$W^2_{2}(\rho_\tau^k, \rho_\tau^{k-1}) \leq 2 \tau (\Phi(\rho_\tau^k) - \Phi(\rho_\tau^{k-1}) ). $$
Since $\Phi$ coincides with $J$, which is bounded, on the sequence $(\rho^k_\tau)_k$, then we have $W^2_{2}(\rho_\tau^k, \rho_\tau^{k-1}) \leq C\tau$ (discrete H\"older behavior), but we also have, if we sum up over \nolinebreak $k$
\begin{equation}\label{discreteH1}
\sum_{k}\tau\left(\frac{W_{2}(\rho_\tau^k, \rho_\tau^{k-1})}{\tau}\right)^2 \leq 2 \Phi(\rho^0),
\end{equation}
which is the discrete version of an $H^1$ estimate. 
As for $\tilde \rho_\tau(t)$, it is an absolutely continuous curve in the Wasserstein space and its velocity on the time interval $[(k-1)\tau,k\tau]$ is given by the ratio $W_2(\rho^{k-1}_\tau,\rho^{k}_\tau)/\tau$. Hence, the $L^2$ norm of its velocity on $[0,T]$ is given by
\begin{equation}\label{L2normv}
\int_0^T |\tilde{\rho}'_\tau|^2_{W_2}(t) dt=\sum_{k}\frac{W^2_{2}(\rho_\tau^k, \rho_\tau^{k-1})}{\tau}, 
\end{equation}
and, thanks to \eqref{discreteH1}, it admits a uniform bound independent of $\tau$ (here we use the notation $|\sigma'|(t)$ for the metric derivative of a curve $\sigma$ and $|\sigma'|_{W_2}(t)$ means that this metric derivative is computed according to the distance $W_2$). This gives compactness of the curves $\tilde\rho_\tau$, as well as an H\"older estimate on their variations (since $H^1\subset C^{0,1/2}$).

\vspace*{0.2cm}
\begin{lemma} (A priori estimates):\\
We have the following a priori estimates:\\
(i) $\mathbf{v_{\tau}}$ is $\tau$-uniformly bounded in $L^2((0,T), L^2_{\rho_\tau}(\Omega))$.\\
(ii) $p_\tau$ is $\tau$-uniformly bounded in $L^2((0,T), H^1(\Omega))$.\\
(iii) $ \mathbf{E_{\tau}}$ and $ \mathbf{\tilde E_{\tau}}$ are $\tau$-uniformly bounded measures.
\label{apriori}
\end{lemma}

\begin{proof}
(i) We have the following equalities:
$$\begin{array}{lll}
\displaystyle \int_{0}^T \int_{\Omega} \rho_\tau |\mathbf{v_{\tau}}|^2 & = & \displaystyle  \sum_{k} \int_{(k-1) \tau}^{k \tau} \int_{\Omega} \rho_\tau^k |\mathbf{v}_\tau^k|^2\\
& = & \displaystyle \sum_{k} \left( \int_{(k-1) \tau}^{k \tau} dt \right) \left(\int_{\Omega} \rho_\tau^k(x) \dfrac{ |x - \mathbf{t}_\tau^k(x)|^2}{\tau^2} \, dx \right) \vseq\\
& = & \displaystyle  \sum_{k} \tau \dfrac{W_{2}^2(\rho_\tau^{k-1}, \rho_\tau^{k})}{\tau^2} = \dfrac{1}{\tau} \sum_{k} W_{2}^2(\rho_\tau^{k-1}, \rho_\tau^{k}).
\end{array}$$
Thanks to the general estimate \eqref{discreteH1} we get $\displaystyle  \int_{0}^T \int_{\Omega} \rho_\tau |\mathbf{v_{\tau}}|^2 \leq 2 \Phi(\rho^0)$.\\

\noindent (ii) Since we have shown the following decomposition: $\nabla p_\tau = - \nabla D - \mathbf{v_{\tau}}$, we have:
$$
\displaystyle \int_{0}^T \int_{\Omega} \rho_\tau |\nabla p_\tau|^2 \; \leq \; 2 \int_{0}^T \int_{\Omega} \rho_\tau |\mathbf{v_{\tau}}|^2 + 2 \int_{0}^T \int_{\Omega} \rho_\tau |\nabla D|^2 \; \leq \; C.
$$
But $p_\tau =0$ on $[\rho_\tau < 1]$, therefore  $\displaystyle \int_{0}^T \int_{\Omega} |\nabla p_\tau|^2 = \int_{0}^T \int_{\Omega} \rho_\tau |\nabla p_\tau|^2 \leq C$.\\

\noindent (iii) We look at $\mathbf{\tilde E_{\tau}}$ and we use the estimates \eqref{discreteH1} and \eqref{L2normv}.
$$\begin{array}{lll} \displaystyle \int_{0}^T \int_{\Omega} |\mathbf{\tilde E_{\tau}}| & = & \displaystyle \int_{0}^T \int_{\Omega} \tilde \rho_\tau |\mathbf{\tilde v_{\tau}}| \, \leq \, \int_{0}^T \left( \int_{\Omega} \tilde \rho_\tau |\mathbf{\tilde v_{\tau}}|^2 \right)^{\frac{1}{2}} \mathop{\underbrace{ \left( \int_{\Omega} \rho_\tau \right)^{\frac{1}{2}} }}\limits_{= 1} \; \leq \; \int_{0}^T  \left( \int_{\Omega} \tilde \rho_\tau |\mathbf{\tilde v_{\tau}}|^2 \right)^{\frac{1}{2}}\\
& \leq & \displaystyle \sqrt{T} \left( \int_{0}^T \int_{\Omega}  \rho_\tau |\mathbf{v_{\tau}}|^2 \right)^{\frac{1}{2}} \; \leq \; C .
\end{array}$$
Therefore, $ \mathbf{\tilde E_{\tau}}$ is a $\tau$-uniformly bounded measure. The proof for $\mathbf{E_{\tau}}$ is almost the same, estimating $L^1$ norms with $L^2$ norms by Cauchy-Schwartz. \end{proof}

\vspace*{0.2cm}
\begin{lemma} Assume that $\mu$ and $\nu$ are absolutely continuous measures, whose densities are bounded by a same constant C. Then, for all function $f \in H^1(\Omega)$, we have the following inequality: 
\begin{equation}
\displaystyle \int_{\Omega} f \, d (\mu - \nu) \; \leq \; \sqrt{C} \, ||\nabla f||_{L^2(\Omega)} W_{2}(\mu, \nu).
\end{equation}
\label{ineq_geod}
\end{lemma}

\begin{proof} Let $\mu_{t}$ be the constant speed geodesic between $\mu$ and $\nu$, and $\mathbf{w}_{t}$ the velocity field such that $(\mu, \mathbf{w})$ satisfies the continuity equation, and $||\mathbf{w}_{t}||_{L^2(\mu_{t})} = W_{2}(\mu, \nu)$. For all $t$, $\mu_{t}$ is absolutely continuous, and its density is bounded by the same constant $C$ a.e.. Therefore:
\begin{eqnarray*}
\displaystyle \int_{\Omega} f \, d (\mu-\nu) & = & \displaystyle \int_{0}^1 \dfrac{d}{dt} \left( \int_{\Omega} f(x) d \mu_{t}(x) \right) \, dt \; 
= \int_{0}^1 \int_{\Omega} \nabla f \cdot \mathbf{w}_t \, d \mu _t \, dt\\
 & \leq & \displaystyle \left(  \int_{0}^1  \int_{\Omega} |\nabla f |^2 \, d \mu_t \, dt \right)^{\!1/2}  \left(  \int_{0}^1  \int_{\Omega} |\mathbf{w}_t|^2 \, d \mu_t \, dt \right)^{\!1/2} \\
 & \leq & \sqrt{C} \, ||\nabla f||_{L^2(\Omega)} W_{2}(\mu, \nu).
\end{eqnarray*}

\end{proof}

\noindent \textit{Remark:} With the same method, we can also prove: 
$$\displaystyle \int_{\Omega} f \, d (\mu - \nu) \; \leq \; C^{\frac{1}{p}} \, ||\nabla f||_{L^p(\Omega)} W_{q}(\mu, \nu)$$ 
for all $f \in W^{1,p}$ and $q$ such that $\frac{1}{p} + \frac{1}{q} = 1$. More generally, if $\mu,\nu\in L^r(\Omega)$ and $||\mu||_{L^r},\,||\nu||_{L^r}\leq C$, one has  $\displaystyle \int_{\Omega} f \, d (\mu - \nu) \; \leq \; C^{\frac{1}{q'}} \, ||\nabla f||_{L^p(\Omega)} W_{q}(\mu, \nu)$, provided $\frac 1p+ \frac 1q +\frac 1r=1+\frac{1}{qr}.$

\subsection{Proof of the theorem in a domain with no exit}
\label{sec:proofnoexit}

\subsubsection*{Step 1: convergence of $(\tilde \rho_\tau, \mathbf{\tilde E_{\tau}})$ and $(\rho_\tau, \mathbf{E_{\tau}})$}

We have proved that $\tilde \rho_\tau$ and $ \mathbf{\tilde E_{\tau}}$ are $\tau$-uniformly bounded measures, thus there exists $(\rho, \mathbf{E})$ such that $(\tilde \rho_\tau, \mathbf{\tilde E_{\tau}})$ converges narrowly to $(\rho, \mathbf{E})$. Let us prove that $(\rho_\tau, \mathbf{E_{\tau}})$ converges to the same limit as $(\tilde \rho_\tau, \mathbf{\tilde E_{\tau}})$. 

We start from the $\rho-$part. The curves  $\tilde\rho_\tau$ actually converge uniformly in $[0,T]$ with respect to the $W_2-$distance. The curves $\rho_\tau$ and  $\tilde\rho_\tau$ coincide on every time of the form $k\tau$. The former is constant on every interval $](k-1)\tau,k\tau]$, whereas the latter is uniformly H\"older continuous  of exponent $1/2$, which implies $W_{2}(\tilde \rho_\tau (t), \rho_\tau (t)) \leq C \tau^{\frac{1}{2}}$. This proves that $\rho_\tau$ converges uniformly to the same limit as $\tilde\rho_\tau$.
 
We now consider a function $f \in \mathcal{C}^{\infty}_{c}([0,T] \times \Omega)$, and prove that $\displaystyle \int_{0}^T \int_{\Omega} f \big(\mathbf{\tilde E_{\tau}} - \mathbf{E_{\tau}}\big)$ converges to $0$ as $\tau$ tends to $0$. We have: $\tilde \rho_\tau(t,.) =  \mathbf{T_{t}}_{\#} \rho_\tau^{k} $ where 
$$\mathbf{T_{t}} = (t - (k-1) \tau) \mathbf{v}_\tau^k + \mathbf{t_{\tau}^k}.$$
Therefore 
$$
\tilde \rho_\tau(t+h,.) = (\mathbf{T_{t}} + h \mathbf{v}_\tau^k )_{\#} \rho_\tau^{k} = ((\mathbf{id} + h \mathbf{v}_\tau^k \circ \mathbf{T_{t}}^{-1}) \circ \mathbf{T_{t}})_{\#} \rho_\tau^{k}  = (\mathbf{id} + h \mathbf{v}_\tau^k \circ \mathbf{T_{t}}^{-1})_{\#} \rho_\tau (t,.),$$
which implies that: $\mathbf{t_{\tilde \rho_\tau(t,.)}^{\tilde \rho_\tau(t+h,.)}} = \mathbf{id} + h \mathbf{v}_\tau^k \circ \mathbf{T_{t}}^{-1}$. We can then express $\mathbf{\tilde v_{\tau}}$ explicitely~:
$$
\mathbf{\tilde v_{\tau}}(t,.) = \mathop{\lim}\limits_{h \rightarrow 0} \dfrac{\mathbf{t_{\tilde \rho_\tau(t,.)}^{\tilde \rho_\tau(t+h,.)}} - \mathbf{id}}{h} =  \mathop{\lim}\limits_{h \rightarrow 0} \dfrac{h \, \mathbf{v}_\tau^k \circ \mathbf{T_{t}}^{-1}}{h} = \mathbf{v}_\tau^k \circ \mathbf{T_{t}}^{-1},$$
 and obtain 
 \begin{eqnarray*}
 \displaystyle \int_{\Omega} f(t,x) \, \tilde \rho_\tau(t,x) \, \mathbf{\tilde v_{\tau}}(t,x) \, dx &= &\int_{\Omega} f (t, \mathbf{T_{t}} (x)) \, \rho_\tau^k (x) \, \mathbf{\tilde v_{\tau}}(t,\mathbf{T_{t}}  (x)) \, dx\\
 & =& \int_{\Omega} f (t, \mathbf{T_{t}} (x)) \, \rho_\tau^k (x) \, \mathbf{v}_\tau^k (x) \, dx.
 \end{eqnarray*}
 Hence 
$$\begin{array}{lll} 
\displaystyle \int_{0}^T \int_{\Omega} f \big(\mathbf{\tilde E_{\tau}} - \mathbf{E_{\tau}}\big) & \leq & \displaystyle \sum_{k} \int_{\tau_{k}}^{\tau_{k+1}} \int_{\Omega} |f(t,x) - f (t, \mathbf{T_{t}}(x))| \, |\mathbf{v}_\tau^k(x)| \, \rho_\tau^k(x) \, dx \, dt\\
& \leq & \displaystyle  \sum_{k} \int_{\tau_{k}}^{\tau_{k+1}} \int_{\Omega} \; \textmd{Lip} f \; |x - \mathbf{T_{t}}(x)|  |\mathbf{v}_\tau^k(x)| \, \rho_\tau^k(x) \, dx \, dt  \\
& \leq & \displaystyle  \sum_{k} \int_{\tau_{k}}^{\tau_{k+1}} \int_{\Omega} \; \textmd{Lip} f \; \tau \; |\mathbf{v}_\tau^k(x)|^2 \, \rho_\tau^k(x) \, dx \, dt \; \leq  \; C \; \textmd{Lip} f \; \tau .
\end{array}$$

\subsubsection*{Step 2: existence of the limit velocity}

Let us prove that $\mathbf{E}$ is absolutely continuous with respect to $\rho$. Let $\theta$ be a scalar measure, and $\mathbf{F}$ a vectorial measure: the function 
$$\Theta: (\theta, \mathbf{F}) \mapsto \left\{ \begin{array}{cl}
\displaystyle \int_{0}^T \int_{\Omega} \dfrac{|\mathbf{F}|^2}{\theta} & \textmd{if} \; \mathbf{F} << \theta \; \textmd{a.e.} \, t \in [0,T]\\
+\infty & \textmd{otherwise}
\end{array} \right.$$ 
is l.s.c. for the weak--$\star$ convergence of measures. Since we have shown the $\tau$-uniform bound: 
$$
\displaystyle \int_{0}^T \int_{\Omega} \dfrac{|\mathbf{E_{\tau}}|^2}{\rho_\tau} \leq C,$$
 we have $\Theta (\rho, \mathbf{E}) < + \infty$. Therefore $\mathbf{E}$ is absolutely continuous with respect $\rho$, and there exists $\mathbf{u}(t,.) \in L^2(\rho(t,.))$ such that $\mathbf{E} = \rho \mathbf{u}$. Moreover, $(\rho, \rho \mathbf{u})$ satisfies the (linear) continuity equation, as limit of $(\tilde \rho_{\tau}, \mathbf{\tilde E_{\tau}})$.\\

Let us now prove that $\mathbf{u}(t) \in C_{\rho(t)}$. Let $t_{0} \in (0,T)$, $h>0$, and $q \in H^1_{\rho (t_{0},.)}$. 
By the continuity equation, we have
$$ \int_{t_{0}}^{t_{0}+h} \int_{\Omega} \nabla q (x) \cdot \mathbf{u}_{t} (x) \rho (t,x) \, dx \; = \; \int_{\Omega} \left[ \rho (t_{0},x) - \rho (t_{0}+h,x) \right]  \, q(x) \, dx.$$
Since $\rho (t_{0}, .) = 1$ wherever $q > 0$, and $\rho (t_{0}+h, .) \leq 1$ a.e., $\displaystyle \int_{\Omega} \left[ \rho (t_{0},x) - \rho (t_{0}+h,x) \right]  \, q(x) \, dx \; \leq \; 0$, and we have for a.e. $t_{0}$ 
\begin{eqnarray*} 0 \; \geq \; \dfrac{1}{h} \int_{t_{0}}^{t_{0}+h} \int_{\Omega} \nabla q (x) . \mathbf{u}_{t} (x) \rho (t,x) \, dx  \; \mathop{\longrightarrow}\limits_{h \rightarrow 0} \; &\displaystyle \int_{\Omega} &\nabla q (x) \cdot \mathbf{u} (t_{0}, x) \rho (t_{0}, .) (x) \, dx\\ &=& \int_{\Omega} \nabla q (x) . \mathbf{u} (t_{0}, x) \, dx.
\end{eqnarray*}
Using the same method between $t_{0}-h$ and $t_{0}$, we also obtain the converse inequality. Finally, we have for a.e. $t_{0}$
 \begin{equation}
 \int_{\Omega} \nabla q (x) \cdot \mathbf{u} (t_{0}, x) \, dx \; = \; 0 \quad \textmd{for all} \; q \in H^1_{\rho(t_{0},.)}.
 \label{complementarity}
 \end{equation}

\subsubsection*{Step 3: the limit velocity satisfies: $\mathbf{u} = P_{C_{\rho}} \mathbf{U}$}

We first prove the decomposition: $\mathbf{U} = \mathbf{u}(t,.) + \nabla p(t, .)$ for a.e. t. We have  $\mathbf{E_{\tau}} = \rho_\tau \mathbf{v_{\tau}} = - \rho_\tau (\nabla D + \nabla p_\tau) =  - \rho_\tau \nabla D - \nabla p_\tau$ since $p_\tau = 0$ on $[\rho_\tau < 1]$. Let us prove that $p_\tau$ converges to $p \in H^1_{\rho}$: as $p_\tau \in L^2([0,T], H^1(\Omega))$, there exists $p$ such that $p_\tau$ weakly converges to $p$ in $L^2([0,T], H^1(\Omega))$. We have obviously: $p \geq 0$ a.e., but it is more difficult to show that $p (t,.) =0$ on $[\rho (t) <1]$. We consider the average functions: 
$$
p_\tau^{a,b} = \displaystyle \dfrac{1}{b-a} \int_{a}^b p_\tau(t,.) \, dt \;\hbox{ and }\; p^{a,b} = \displaystyle \dfrac{1}{b-a} \int_{a}^b p(t,.) \, dt.
$$
 Since $p_\tau = 0$ on $[\rho_\tau < 1]$, we have 
 \begin{eqnarray*}
 0 = \displaystyle  \int_{a}^b \int_{\Omega} p_\tau(t,x) (1-\rho_\tau(t,x)) \, dx \, dt \; &=& \; \dfrac{1}{b-a} \int_{a}^b \int_{\Omega} p_\tau(t,x)(1-\rho_\tau(a,x)) \, dx \, dt \\
 && \hspace*{-0.7cm} +   \dfrac{1}{b-a} \int_{a}^b \int_{\Omega} p_\tau(t,x)(\rho_\tau(a,x)-\rho_\tau(t,x)) \, dx \, dt.
 \end{eqnarray*}
The first integral reads: $\displaystyle \int_{\Omega} p_\tau^{a,b}(x) (1 - \rho_\tau(a,x)) \, dx \; \mathop{\longrightarrow}\limits_{\tau \rightarrow 0} \; \int_{\Omega} p^{a,b}(x) (1-\rho(a,x)) \, dx$, as $p_\tau^{a,b}$ weakly converges in $H^1(\Omega)$ -- therefore strongly in $L^2(\Omega)$ -- to $p^{a,b}$, and $\rho_\tau(a,.)$ weakly--$\star$ converges in $L^{\infty}(\Omega)$ to $\rho(a,.)$. Moreover, for every Lebesgue point $a$ of $p(., x)$, we have: $p^{a,b} \, \mathop{\longrightarrow}\limits_{b \rightarrow a} \, p(a,.)$, therefore, for all these $a$, we have
$$\displaystyle \int_{\Omega} p^{a,b}(x)(1-\rho(a,x)) \, dx \, dt \; \mathop{\longrightarrow}\limits_{b \rightarrow a} \; \int_{\Omega} p(a,x) (1 - \rho(a,x)) \, dx.
$$

\noindent Using lemma \ref{ineq_geod}, we obtain for the second integral: \\
\begin{eqnarray*}
&&\displaystyle  \int_{a}^b \int_{\Omega} p_\tau(t,x) \big(\rho_\tau(a,x)-\rho_\tau(t,x)\big) \, dx \, dt\\
 & \leq & \displaystyle  \int_{a}^b ||\nabla p_\tau(t,.)||_{L^2(\Omega)} \mathop{\underbrace{W_{2}(\rho_\tau(a,.), \rho_\tau(t,.))}}\limits_{\leq C \sqrt{b-a}} \, dt \\
& \leq & \displaystyle C \sqrt{b-a} \left( \int_{a}^b ||\nabla p_\tau(t,.)||_{L^2(\Omega)}^2 \, dt \right)^{\frac{1}{2}}  \left( \int_{a}^b dt \right)^{\frac{1}{2}}\\
& \leq & \displaystyle C (b-a)  \left( \int_{a}^b ||\nabla p_\tau(t,.)||_{L^2(\Omega)}^2 \, dt \right)^{\frac{1}{2}} .
\end{eqnarray*}
As $\displaystyle  \int_{0}^T ||\nabla p_\tau(t,.)||_{L^2(\Omega)}^2 \, dt$ is $\tau$-uniformly bounded, $||\nabla p_\tau(t,.)||_{L^2(\Omega)}^2$ weakly converges to a measure $\mu$. Therefore, beyond a zero measure set of points $a$, we have 
$$\displaystyle \mathop{\lim}\limits_{\tau \rightarrow 0} \dfrac{1}{b-a} \int_{a}^b \int_{\Omega} p_\tau(t,x)(\rho_\tau(a,x)-\rho_\tau(t,x)) \, dx \, dt \; \leq C \sqrt{\mu([a,b])} \; \mathop{\longrightarrow}\limits_{b \rightarrow a} 0 .$$
We finally obtain: $\displaystyle \int_{\Omega} p(a,x) (1 - \rho(a,x)) \, dx = 0$ for almost every $a$.\\

Hence $\mathbf{E} = - \rho \nabla D - \nabla p$, with $p = 0$ on $[\rho < 1]$, so: $\mathbf{E} = - \rho (\nabla D + \nabla p)$. Since: $\mathbf{E} = \rho \mathbf{u}$, we have shown the following decomposition: 
$$
\mathbf{u} = - \nabla D - \nabla p\;\;\hbox{i.e.} \;\; \mathbf{U} = \nabla p + \mathbf{u}.$$
Moreover, by Equality (\ref{complementarity}), $\mathbf{u}$ and $\nabla p$ satisfy the complementarity relation
$$\int_{\Omega} \nabla p (t,x) \cdot \mathbf{u} (t,x) \, dx \; = \; 0 \quad \textmd{ for a.e. } t ,$$
which implies that we have exactly: $\mathbf{u (t)} = P_{C_{\rho (t)}} \mathbf{U}$.


\section{Proof of the theorem in the general case}
\label{sec:exexit}
We consider here the general case where $\Gamma_{out} \not = \emptyset$.

\subsection{Lack of geodesic convexity}

The main problem we encounter when we want to generalize the previous proof is the fact that the classical geodesics no longer belong to the admissible space $K$, which is no more a geodesically convex set. Indeed, if we consider a density $\rho^{0}$ which is constant equal to $1$ on a subset of $\Omega$, a measure $\rho^{1}$ which is concentrated on $\Gamma_{out}$, and the geodesic $\rho(t,.)$ between them, the density of $\rho(t,.)$ will be of the order of $1/(1-t)$ where it is positive, and therefore $\rho(t,.) \not \in K$ for all $t \in \, ]0,1[$. 

This is one of the main reasons that prevent from using the standard theory of gradient flow for geodesically convex functionals in the Wasserstein space (see \cite{Amb}).

In this section we will investigate the connectedness properties of the set $K$. For the sake of this work, we will see that we need to estimate the length to connect two measures in $K$ at a very single point of the proof. Yet, we think that these estimates are interesting in themselves and this is why we try to present them so that they will be valid in any dimension $d$.

We define a new distance, coming from a minimal length approach, on $K$:

\begin{proposition} (Continuity of the length L)
For $\mu, \nu \in K$, we define the length 
\begin{equation}
\label{eq:defL}
L(\mu, \nu) = \inf \left\{\displaystyle \int_{0}^1 |\sigma'|_{W_2}(t) dt \;: \; \sigma(t) \in K, \; \sigma(0) = \mu, \; \sigma(1) = \nu \right\}.
\end{equation}
This length is finite, and it is a distance on $K$ which is continuous for the narrow convergence: if $(\mu_{n}), (\nu_{n})$ are sequences that narrowly converge in $K$ to $\mu$ and $\nu$, then $L(\mu_{n}, \nu_{n})$ converges to $L(\mu, \nu)$.
\label{contL}
\end{proposition}

To prove this proposition, we will first analyze the case were the domain $\Omega$ is the unit cube and the door is one of the sides. We set $Q= \, ]0,1[^{d-1} \times \, ]-1,0[$, $\overline{Q}=[0,1]^{d-1}\times[-1,0]$ and $S=[0,1]^{d-1}\times \{0\}$. We will still denote by $K$ the set of admissible measures, i.e. those who are composed by a density less than $1$ in $Q$ and by a possibly singular part on $S$. We will denote by $y$ the last component of a point $(x,y)\in \R^d =\R^{d-1}\times \R$. When integrating over $S$, we write $dx$ instead of $\haus^{d-1}(dx)$ or similar expressions.

Let us start from a simpler case.

A first useful lemma is the following: 
\begin{lemma}\label{previous}
Let $\rho^0,\rho^1$ be two probability measures on $\overline{Q}$ of the form $\rho^i=\rho^i_Q+\rho^i_S$, where $\rho^i_Q$ has a density on $Q$ bounded by $\overline{k}$ and $\rho^i_S$ is concentrated on $S$. Set $\DD=W_1(\rho^0,\rho^1)$. Then, for any Lipschitz continuous function $j$ we have
\begin{gather*}
\int_S  jd(\rho^0_S-\rho^1_S)\leq   Lip(j)\DD+c(\overline{k})||j||_{L^\infty}\DD^{1/2},\\
\int_Q j(\rho^0_Q-\rho^1_Q)\leq  2Lip(j)\DD+c(\overline{k})|||j||_{L^\infty}\DD^{1/2}.
\end{gather*}
\end{lemma}

\begin{proof} We start from the first estimate: consider a function $\chi_\delta:\overline Q\to [0,1]$ such that $\chi_\delta=1$ on $S$, $\chi_\delta=0$ outside a strip of width $\delta$ from $S$, and $|\nabla\chi_\delta|\leq \delta^{-1}$ (as a matter of fact, it defines this function as $\chi_\delta(x,y)=(1+\delta^{-1} y)_+$). We may write
$$\int_S  jd(\rho^0_S-\rho^1_S)=\int_{\overline{Q}} j\chi_\delta d(\rho^0-\rho^1)-\int_Q j\chi_\delta d(\rho^0_Q-\rho^1_Q)\leq Lip(j\chi_\delta) \DD+\overline{k}\delta||j||_{L^\infty}.$$
Then we use $Lip(j\chi_\delta)\leq Lip(j)+||j||_{L^\infty}\delta^{-1}$ and we get
$$\int_S  jd(\rho^0_S-\rho^1_S)\leq \left( Lip(j)+\frac{||j||_{L^\infty}}{\delta}\right)\DD+\overline{k}\delta||j||_{L^\infty},$$
which implies, by choosing $\delta=\DD^{1/2}$,
$$\int_S  jd(\rho^0_S-\rho^1_S)\leq   Lip(j)\DD+c(\overline{k})||j||_{L^\infty}\DD^{1/2}.$$
As far as the second estimate is concerned, just write 
$$\int_Q j(\rho^0_Q-\rho^1_Q)=\int_{\overline{Q}} j(\rho^0-\rho^1)-\int_S  jd(\rho^0_S-\rho^1_S)$$
and use $\displaystyle \int_{\overline{Q}} j(\rho^0-\rho^1)\leq Lip(j)\DD$ and the previous inequality. \end{proof}

It is important to notice in the above inequality that, once we fix $\rho^i_Q$ or  $\rho^i_S$, the two estimates separately make $\DD$ appear, where $\DD$ may be the $W_1$ distance between any pair of measures, satisfying the constraints, having  $\rho^i_Q$ or  $\rho^i_S$ as an internal or boundary part, respectively. The pair of measures we use need not to be the same in the two estimates.

\begin{lemma}\label{next}
Let $\rho^0,\rho^1\in K$ be two admissible probability measures on $\overline{Q}$  and $L,M\geq 1$. Suppose that $\rho^0$ and $\rho^1$ are of the following form:
$$\rho^i=\rho^i_Q+\rho^i_S,\quad \rho^i_Q\ll \lcal^d,\;\rho^i_S=h_i\cdot\haus^{d-1},\;h_i\leq M,\; Lip(h_i)\leq L,\quad i=0,1.$$

Then, there exists a curve $\rho^t$ from $\rho^0$ to $\rho^1$, contained in $K$ (the set of admissible measures) and such that its $W_2-$length does not exceed $C(d)M^{1/2}\sqrt{L \ell + M \ell^{1/2}}$, where we set $\ell:=W_1(\rho^0,\rho^1)$.

Moreover, the same stays true if $\ell$ stands for a number such that there exist ``extensions'' of $\rho^i_Q$ on $S$ and of $\rho^i_S$ on $Q$ that belong to $K$ and such that for both extensions the new $W_1-$distance is smaller than $\ell$ (but the two extensions may be different). If instead of staying in $K$ the constraint on the density in $Q$ is relaxed to ``being smaller than $\overline{k}$'' with $\overline{k}>1$, the constant will also depend on $\overline{k}$, as in Lemma \ref{previous}.

\end{lemma}

\begin{proof} It is possible to replace the two probabilities on  $\overline{Q}$  with probabilities $\tilde{\rho}^i$ on $R=[0,1]^{d-1}\times[-1,M]$ so that $\tro^i$ is absolutely continuous with density less than $1$ and $(\pi_Q)_\#\tro^i=\rho^i$ (where $\pi_Q$ is the projection on $Q$). We will take
$$\tro^i=\rho^i_Q+\ONE_{y<h_i(x)}\cdot\lcal^d.$$
Consider the geodesic $\tro^t$ from $\tro^0$ to $\tro^1$. It is a curve of measure whose length is exactly $W_2(\tro^0,\tro^1)$. Moreover, if one projects on $Q$ all the trajectories of the particles of this curve, one gets the curve $(\pi_Q)_\#\tro^t$, which connects $\rho^0$ to $\rho^1$ but stays in $K$ (since the only effect of the projection is to send all the mass on $R\setminus Q$ on $S$, while the densities inside $Q$ are not affected. And we know that the densities of $\tro^t$ will not be larger than $1$, since this is the case for a geodesic between two measures with densities bounded by $1$.

Hence we only need to estimate $W_2(\tro^0,\tro^1)$. For simplicity, let us estimate $W_1$ instead of $W_2$. We will conclude by the inequality $W_2\leq (\diam R)^{1/2}W_1^{1/2}$. Notice that the diameter of $R$ is $\sqrt{(M+1)^2+d-1}\leq C(d)M$.

To estimate $W_1$, take a function $f\in Lip_1(R)$. What follows will be easier to justify in case $f$ is regular but everything will work (by density, or instance), for any $f$ whose Lipschitz constant does not exceed $1$. Let us define, for $x\in  [0,1]^{d-1}$ and $a,b\in [0,M]$, 
$$g(x,a,b)=\frac{\int_a^b f(x,t)dt}{b-a}.$$
We denote by $g_x,\,g_a$ and $g_b$ the partial derivatives of $g$. We can verify that
$$|g_x(x,a,b)|=\frac{\left|\int_a^b f_x(x,t)dt\right|}{|b-a|}\leq Lip(f)=1,$$
then we compute $g_b$ and we get
$$|g_b(x,a,b)|=\left|\frac{ f(x,b)}{b-a}-\frac{g(x,a,b)}{b-a}\right|\leq \frac{Lip(f)}{|b-a|}{|b-a|}=1,$$
and, analogously, $|g_a|\leq 1$.

In particular, if one takes two Lipschitz functions $a(x)$ and $b(x)$, one has $Lip(g(x,a(x),b(x)))\leq 1+Lip(a)+Lip(b)$.

Now we write 
$$\int_R f\,d(\tro^0-\tro^1)=\int_{Q}f\,d(\rho^0_Q-\rho^1_Q)+\int_S \left(\int_0^{h_0(x)}f(x,t)dt-\int_0^{h_1(x)}f(x,t)dt\right)dx.$$
We estimate both terms thanks to the previous lemma.  
The first term in the right hand side is less than $\ell$, while for the second we may write 
$$
\int_S \left(\int_0^{h_0(x)}f(x,t)dt-\int_0^{h_1(x)}f(x,t)dt\right)dx
=\int_S g(x,h_0(x),h_1(x))(h_0-h_1)(x) dx.
$$
Hence we are in the case of the previous lemma with $j(x)=g(x,h_0(x),h_1(x))$, and hence $Lip(j)\leq 1+2L$ and $||j||_{L^\infty}\leq M+\sqrt{d}$ (the first estimate comes from our study of $g$, for the second just suppose that $f$ vanishes somewhere on $S$).

Hence we get, using the arbitrariness of the function $f$
$$W_1(\tro^0,\tro^1)\leq \ell+(1+2L)\ell+2M\ell^{1/2}.$$

To simplify the computations we use $1\leq M,L$ and get 
$$W_2(\tro^0,\tro^1)\leq C(d)M^{1/2}W_1(\tro^0,\tro^1)^{1/2}\leq C(d)M^{1/2}\sqrt{L\ell+M\ell^{1/2}}.$$

The last part of the statement is an easy consequence of the technique we used and of Lemma \ref{previous}. \end{proof}

\begin{theorem}
Let $\mu_0$ and $\mu_1$ be two probabilities in $K$. Then there exists a curve $(\mu_t)_t$ connecting $\mu_0$ to $\mu_1$, such that its $W_2-$length does not exceed $C(d)W_1(\mu_0,\mu_1)^{1/(4d)}$ and that $\mu_t\in K$ for every $t$.
\end{theorem}

\begin{proof} Take $\ve>0$ and modify $\mu^i$ into a new measure $\rho^i\in K$ by regularizing in the direction of $x$: it is sufficient to take the convolution of $\mu^i_S$ with a kernel of the form $C(\ve^{1-d}-\ve^{-d}|x|)_+$. This ensures that the $W_1$ distance has not increased and that  the new measures  on $S$ will have Lipschitz and bounded densities on $S$, with $M\leq C\ve^{1-d}$ and $L\leq C\ve^{-d}$, and on $Q$ the constraint is kept as well. Yet, there is a problem: these measures may exit the domain. There are two possible ways for solving this problem, and both will be useful. 

One possibility is rescaling of a factor $(1+2\ve)^{-1}$, so that all the mass is pushed again into the domain. This does not change significantly the values of $M$ and $L$ but the densities inside will be no more bounded by $1$. They will be bounded by a constant $\overline{k}$ close to $1$. In this case too the Wasserstein distance has not increased, since the rescaling was a contraction.

The other possibility is composing with a contracting transport $T:\overline{Q}_{\ve}\to \overline{Q}$ ($\overline{Q}_{\ve}$ being the $\ve-$neighborhood of $\overline{Q}$), which is chosen so that the convolution of the constant function $1$ is sent onto the constant function $1$ (this is possible thanks to the fact the convolution keeps the mass unchanged). This construction ensures that the constraint inside $Q$ will be satisfied but unluckily, since the inverse of $T$ is not Lipschitz continuous (due to the fact that the densities vanished on the boundary of $\overline{Q}_{\ve}$), it is not suitable for $S$. Anyway, in this case too, the Wasserstein distance was not increased.

Hence we do a mixed procedure: we use the second possibility in $Q$ and the first on $S$. It is clear that in this way we have good densities both in $Q$ and on $S$, and we can apply Lemma \ref{previous} and the last statement of Lemma \ref{next}. Notice that the $W_1-$distance between the two measures $\rho^i\in K$ that we constructed could be larger than $\ell$. It is easily estimated by something like $\ell+\ve$ but this is not sufficient for the following estimates.

Now, to connect $\mu^0$ to $\mu^1$, one can first connect each $\mu^i$ to $\rho^i$, and the cost is no more than $\ve$, since it is sufficient to spread every particle on a ball of radius $\ve$ (i.e. the radius of the support of the previous kernel) when we do convolution, and then to move it no more than $\ve$ when we compose with a contraction. After that, one uses the previous Lemma to estimate the length for connecting $\rho^0$ to $\rho^1$ and gets
$$\min length\leq 2\ve+ C(d)\ve^{(1-d)/2}\sqrt{\ve^{-d}\ell + \ve^{1-d} \ell^{1/2}}=2\ve+C(d)\ve^{\frac 32 -d}\sqrt{\frac{\ell}{\ve^2}+\frac{\ell^{1/2}}{\ve}},$$
where $\ell$ denotes the $W_1$ distance between $\mu^0$ and $\mu^1$.
If one supposes that $\ve$ is chosen so that $\ell\ve^{-2}$ is smaller than $1$, one can estimate the last sum in the square root and get
$$\min length\leq 2\ve+C(d)\ve^{1-d}\ell^{1/4}.$$
Choosing $\ve=\ell^{1/(4d)}$ gives at the same time that $\ell\ve^{-2}$ is small and that the minimal length may be estimated by $\ell^{1/(4d)}$.
\end{proof}

To approach the general case one needs to use the following theorem, which has already been used in a transport-related setting with density constraints in the variational theory of incompressible Euler equations by Y. Brenier and provides a useful tool for reducing to the cube
 (see Refs \cite{DacMos} and \cite{Roesch} for the applications to fluid mechanics).
\begin{theorem}
For any sufficiently good domain $\Omega\subset\R^d$ which is homeomorphic to the cube, there exists a bi-lipschitz homeomorphism $\phi:\overline{\Omega}\to \overline{Q}$ such that $\phi_\#(\lcal^d_{|\Omega})=c\lcal^d_{|Q}$. Moreover, the behavior of $\phi$ on the boundary may be prescribed at will.
\end{theorem}

\subsection{Generalization of the technical lemmas}

In this section, we briefly explain how to generalize the technical lemmas we used in the first proof (with $\Gamma_{out} = \emptyset$).

\paragraph{Conditions on the minimizer for the discrete problem.}
\begin{itemize}
\item First of all, in lemma \ref{ineq_kant}, we can't prove the uniqueness of the minimizer $\rho_{m}$ with the same method. Indeed, exactly as we explained in the previous section for geodesics, the interpolation $\rho_{t}$ between two possible minimizers does not necessarily belong to $K$. Therefore, we will have to apply the ``selection'' method explained in Remark \ref{selection} in order to prove inequality (\ref{ineq_kant_form}).
More precisely, the case where $\bar \rho > 0$ remains unchanged, but in the general case, we fix a minimizer $\rho_{m}$ of $\phi$, and we define $\rho_{m, \delta}$ as a minimizer of 
$$
\phi_{\delta}(\rho):= \displaystyle \int_{\Omega} D \rho + I_{K}(\rho) + \dfrac{1}{2 \tau} W_{2}^2(\rho, \bar \rho_{\delta}) + c_{\delta} W_{2}^2(\rho, \rho_{m}),$$
with $c_{\delta}$
 that converges to $0$ slower than $W_{2}(\bar \rho, \bar \rho_{\delta})$. Since $\rho_{m}$ and $\rho_{m, \delta}$ are minimizers of $J$ and $\phi_{\delta}$, we have the following inequalities:
$$\left\{ \begin{array}{rcl}
\displaystyle \int_{\Omega} D \rho_{m, \delta} + \dfrac{1}{2 \tau} W_{2}^2(\rho_{m, \delta}, \bar \rho_{\delta}) + c_{\delta} W_{2}^2(\rho, \rho_{m}) & \leq & \displaystyle \int_{\Omega} D \rho_{m} + \dfrac{1}{2 \tau} W_{2}^2 (\rho_{m}, \bar \rho_{\delta})\\
\displaystyle \int_{\Omega} D \rho_{m} + \dfrac{1}{2 \tau} W_{2}^2 (\rho_{m}, \bar \rho) & \leq & \int_{\Omega} D \rho_{m, \delta} + \dfrac{1}{2 \tau} W_{2}^2 (\rho_{m, \delta}, \bar \rho)
\end{array} \right.$$
which implies, using the triangular inequality:
$$\begin{array}{rcl}
W_{2}^2(\rho_{m, \delta}, \rho_{m}) & \leq & \dfrac{1}{2 \tau c_{\delta}} \left[ W_{2}^2 (\rho_{m, \delta}, \bar \rho) - W_{2}^2(\rho_{m, \delta}, \bar \rho_{\delta})  + W_{2}^2 (\rho_{m}, \bar \rho_{\delta}) - W_{2}^2 (\rho_{m}, \bar \rho)  \right]\\
 & \leq & \dfrac{C}{2 \tau c_{\delta}} W_{2}(\bar \rho, \bar \rho_{\delta}) \; \mathop{\longrightarrow}\limits_{\delta \rightarrow 0} \; 0 . 
\end{array} $$
Therefore, $\rho_{m, \delta}$ converges to the fixed minimizer $\rho_{m}$. We can then pass to the limit $\delta \rightarrow 0$, and obtain the same inequality for $\rho_{m}$.\\

\item It is also useful to notice another feature of the problem with an exit $\Gamma_{out}$: once some mass arrives to the exit, it does not move anymore. This precisely means the following: if $\gamma^k_\tau$ is an optimal transport plan from the selected measure $\rho^k_\tau$ to the previous one, $\rho_\tau^{k-1}$, and $(x,y)\in\spt(\gamma^k_\tau)$ with $y\in\Gamma_{out}$, then $y=x$. This means that all the mass which was already on the door for $\rho^{k-1}_\tau$ will not move. To prove it, it is sufficient to consider the map $F:\Ov\times\Ov\to\Ov\times\Ov$ defined by $F(x,y)=(y,y)$ if $y\in\Gamma_{out}$ and $F(x,y)=(x,y)$ if $y\notin \Gamma_{out}$. The measure $F_\#\gamma^k_\tau$ is a transport plan between a new measure $\rho$ and $\rho^{k-1}_\tau$, which reduces the transport cost and the functional $J$ (since $D$ is minimal on the exit). Moreover, since $\rho$ is obtained from $\rho_\tau^k$ by moving some mass onto the door, we have $\rho\in K$ as well. This would contradict the optimality of $\rho^k_\tau$ unless $F_\#\gamma^k_\tau=\gamma^k_\tau$, which is the thesis. 

This also proves uniqueness of the optimal transport plan between  $\rho^k_\tau$ and $\rho_\tau^{k-1}$ since, if we look it the other way around (from $\rho^{k-1}_\tau$ to $\rho_\tau^{k}$), we can decompose the problem in one part which will not move (corresponding to $\rho^{k-1}_\tau\ONE_{\Gamma_{out}}$) and one part which is the transport of an absolutely continuous density ($\rho^{k-1}_\tau\ONE_{\Oa}$
). 

We will also denote by $E^k_\tau$ the excess mass of $\rho^k_\tau$ with respect to $\rho^{k-1}_\tau$ on the exit, i.e. $E^k_\tau:=\rho^k_\tau(\Gamma_{out})-\rho^{k-1}_\tau(\Gamma_{out})\geq 0$.\\
 
\item In lemma  \ref{decomposition},  the solution  of $$ \rho_\tau^{k} \, \in \, \mathop{\textmd{argmin}}\limits_{\rho \in K} \left\{ \displaystyle \int_{\Omega} D(x) \rho (x) \, dx + \dfrac{1}{\tau}  \int_{\Omega} \bar\varphi(x) \rho (x) \, dx \right\}$$ is not necessarily the same in the general case, as there exists no limit density on $\Gamma_{out}$. Let us define: $l:= \mathop{\inf}\limits_{x \in \Gamma_{out}} F(x)$, and $\Gamma_{min} = \{ x \in \Gamma_{out}: F(x) = l\}$. If $|[ F < l ]| \geq 1$, then the solution is the same as in the previous proof. However, if $|[ F < l ]| < 1$, it costs less to put a part of the density onto $\Gamma_{out}$. The solution is therefore given by:
$$\left\{ \begin{array}{ccl}
\rho_\tau^{k} = 1 & \textmd{on} & [F < l],\\
\rho_\tau^{k} > 0 & \textmd{on} & \Gamma_{min}, \; \textmd{with} \; \rho_{\tau}^k (\Gamma_{min}) = 1 - |[ F < l ]|,\\
\rho_\tau^{k} \leq 1 & \textmd{on} & [F = l] \backslash \Gamma_{min},\\
\rho_\tau^{k} = 0 &  \textmd{on} & [F > l] .\\
\end{array} \right.$$
The pressure $p_{\tau}^k$ defined by 
$$
p_{\tau}^k(x):= (l - F(x))_{+} = \left( l - D(x) - \dfrac{\bar\varphi(x)}{\tau} \right)_{+}
$$
 then belongs to $H^1_{\rho_{\tau}^k}$, and we prove the decomposition $\mathbf{U} = \mathbf{v}_\tau^k + \nabla p_{\tau}^k$ as before.\\
\end{itemize}

In order to prove the a priori estimates of lemma \ref{apriori}, we have to take into account the singularity part of the densities on $\Gamma_{out}$. Notice that, to avoid any ambiguity where the transport does not exist,  we only defined a discrete velocity vector field inside $\mathring{\Omega}$. To be clearer, we want to spend some disambiguation words on what $\mathbf{\tilde E_{\tau}}$ and $\mathbf{E_{\tau}}$ are in this case.  
\begin{itemize}
\item The measure $\mathbf{\tilde E_{\tau}}$ is as usual defined as the vector measure satisfying the continuity equation with the curve $\tilde\rho_\tau$. We also have an explicit formula in terms of the optimal transport plans $\gamma^k_\tau$ from $\rho^{k}_\tau$ to $\rho^{k-1}_\tau$: for any $t\in [k-1\tau,k\tau[$ take
$$\mathbf{\tilde E_{\tau}}(t):=\big(\pi_{(k\tau-t)/\tau}\big)_\#\left(\frac{x-y}{\tau}\cdot  \gamma^k_\tau\right),$$
where $\pi_s(x,y)=(1-s)x+sy$.
\item The measure $\mathbf{E_{\tau}}$ is simply defined as the product of $\rho_\tau\ONE_{\Oa}$ times the velocity vector field defined in Section 2.3, on the non-singular part only  (again we use $\mathring{\Omega}$ instead of $\Omega$ to stress that  the boundary is excluded). As before, the idea is that this vector measure satisfies good properties from optimality conditions, while the previous one satisfies the continuity equation. We need to compare them.
\item There is also in this case a third vector measure, that we can call $\mathbf{\hat E_{\tau}}$, which is defined exactly as $\mathbf{\tilde E_{\tau}}$ but ignoring the part on $\Gamma_{out}$:
$$\mathbf{\hat E_{\tau}}(t):=\big(\pi_{(k\tau-t)/\tau}\big)_\#\left(\frac{x-y}{\tau}\cdot  \ONE_{x\in\Oa} \gamma^k_\tau\right).$$
The utility of  $\mathbf{\hat E_{\tau}}$ is that it is more easily comparable to $\mathbf{E_{\tau}}$.
\end{itemize}

We come back to the proof of lemma \ref{apriori}: as a matter of fact, we now have: 
$$
\displaystyle W_{2}^2 (\rho_{\tau}^{k-1}, \rho_{\tau}^k) = \tau^2 \int_{\Omega} \rho^k_\tau |\mathbf{v^k_{\tau}}|^2 + \int_{\Gamma_{out} \times \Omega} |x-y|^2 d \gamma_{\tau}^k (x,y),
$$
 where $\gamma_{\tau}^k$ is the optimal transport plan between $\rho_{\tau}^{k}$ and $\rho_{\tau}^{k-1}$. Therefore, we have 
 $$
 \displaystyle \int_{\Omega} \rho_\tau |\mathbf{v_{\tau}}|^2 \leq \tau^{-2}W_{2}^2 (\rho_{\tau}^{k-1}, \rho_{\tau}^k),
 $$ 
 and the a priori estimates (i) and (ii) are still satisfied.
The proof of (iii) is unchanged, but let us remark that we have no longer the equality: 
$$\displaystyle \int_{\Omega} \tilde \rho_{\tau}^k |\mathbf{\tilde v_{\tau}^k }|^2 = \int_{\Omega} \rho_{\tau}^k |\mathbf{v_{\tau}^k }|^2,
$$
 and that the geodesic $\tilde \rho_{\tau}$ does not belong to $K$.\\

Lemma \ref{ineq_geod} is no longer true for densities that are not absolutely continuous with respect to the Lebesgue measure. Indeed, as we have seen before, the geodesic between two densities of $K$ does not belong to $K$. We prove instead the following lemma:

\vspace*{0.2cm}
\begin{lemma} Let $\mu, \nu \in K$. Then, for all function $f \in H^1$ with $f=0$ on $\Gamma_{out}$, we have the following inequality: 
$$\displaystyle \int_{\Omega} f \, d (\mu - \nu) \; \leq \; \, ||\nabla f||_{L^2(\Omega)} L(\mu, \nu)$$
\label{ineq_length}
where $L(\mu,\nu)$ is the length of the shortest path in $K$ joining $\mu$ and $\nu$ (see~(\ref{eq:defL})).
\end{lemma}

\begin{proof} The proof is an adaptation of the one of Lemma  \ref{ineq_geod} : let $\sigma_t$ be a minimal length curve in $K$ joining $\mu$ and $\nu$, and let $\mathbf{w}_t$ such that $(\sigma, \mathbf{w})$ satisfies the continuity equation and $||\mathbf{w}_{t}||_{L^2(\sigma_{t})} = L(\mu, \nu)$. Since $f \in H^1$ with $f=0$ on $\Gamma_{out}$, then $\nabla f$ does not see the part of $\sigma\mathbf{w}$ on the boundary), so that we have:
\begin{eqnarray*}
\displaystyle \int_{\Omega} f \, d (\mu-\nu) & = & \displaystyle \int_{0}^1 \dfrac{d}{dt} \left( \int_{\Omega} f d \sigma_t \right) \; = \; \int_{0}^1 \int_{\mathring{\Omega}} \nabla f \cdot \mathbf{w}_t \, d \sigma _t \, dt\\
 & \leq & \displaystyle \left(  \int_{0}^1  \int_{\mathring{\Omega}} |\nabla f |^2 \, d \sigma_t \, dt \right)^{\!1/2}  \left(  \int_{0}^1 \int_{\mathring{\Omega}} |\mathbf{w}_t|^2 \, d \sigma_t \, dt \right)^{\!1/2} \\
 & \leq & ||\nabla f||_{L^2(\Omega)} L(\mu, \nu)
\end{eqnarray*}
since $\sigma_t \leq 1$ in $\mathring{\Omega}$. \end{proof}

\subsection{Generalization of the proof}

At step 1, we need again to prove that the limits of $\mathbf{\tilde E_{\tau}}$ and $\mathbf{E_{\tau}}$ are the same. As far as the limits of $\tilde\rho_\tau$ et $\rho_\tau$ are concerned, everything works as in Section~\ref{sec:proofnoexit}: this also proves that the limit curve $\rho$ belongs to $K$, since this is the case for $\rho_\tau$ (but not for $\tilde\rho_\tau$). 

It is easy to check that the comparison we did in Step 1 of Section~\ref{sec:proofnoexit} may be performed again so as to obtain that the limit of $\mathbf{\hat E_{\tau}}$ and $\mathbf{ E_{\tau}}$ are the same. What we need to do now is proving that the limit of $\mathbf{\hat E_{\tau}}$ and $\mathbf{\tilde E_{\tau}}$ are the same. We will prove that the mass of $\mathbf{\tilde E_{\tau}}-\mathbf{\hat E_{\tau}}$ is negligible, i.e. that 
$$\displaystyle \int_{0}^T dt\int_{\Ov} d\big| \mathbf{\tilde E_{\tau}}(t)-\mathbf{\hat E_{\tau}}(t)\big| \; \mathop{\longrightarrow}\limits_{\tau \rightarrow 0} \; 0.$$ To do this, it is sufficient to estimate
$$\sum_{k=0}^{T/\tau}\int_{(k-1)\tau}^{k\tau}dt\int_{\Ov\times\Ov}\frac{|x-y|}{\tau}\ONE_{\Gamma_{out}\times\Ov}\,d\gamma^k_\tau=\sum_{k=0}^{T/\tau}\int_{\Ov\times\Ov}|x-y|\ONE_{\Gamma_{out}\times\Ov}\,d\gamma^k_\tau.$$
Thanks to what we underlined before, namely that the mass which is on $\Gamma_{out}$ does not move any more, we know that $|x-y|\ONE_{\Gamma_{out}\times\Ov}\,d\gamma^k_\tau=|x-y|\ONE_{\Gamma_{out}\times\Oa}\,d\gamma^k_\tau$ and the mass of $\ONE_{\Gamma_{out}\times\Oa}\,d\gamma^k_\tau$ is exactly the excess mass $E^k_\tau$. Thanks to the Lemma \ref{EW23} below, we can go on and obtain
\begin{multline*}
\sum_{k=0}^{T/\tau}\int_{\Ov\times\Ov}|x-y|\ONE_{\Gamma_{out}\times\Ov}\,d\gamma^k_\tau=\sum_{k=0}^{T/\tau}\int_{\Ov\times\Ov}|x-y|\ONE_{\Gamma_{out}\times\Oa}\,d\gamma^k_\tau\\
\leq \sum_{k=0}^{T/\tau}\left( \int_{\Gamma_{out} \times \Oa} |x - y|^2 d \gamma_{\tau}^k \right)^{\frac{1}{2}} \left( \int_{\Gamma_{out} \times \Oa} d \gamma_{\tau}^k \right)^{\frac{1}{2}} \leq \sum_{k=0}^{T/\tau}
W_2(\rho^{k-1}_\tau,\rho^k_\tau)^{4/3}\\
\leq\left( \sum_{k=1}^{E(T/\tau)} W_{2}(\rho_{\tau}^{k-1} , \rho_{\tau}^k )^2 \right)^{\frac{2}{3}} \left( \sum_{k=1}^{E(T/\tau)} 1 \right)^{\frac{1}{3}}
\leq  (C \tau)^{\frac{2}{3}} \left( \dfrac{T}{\tau} \right)^{\frac{1}{3}} \; =  \;C \; \tau^{\frac{1}{3}} \; \mathop{\longrightarrow}\limits_{\tau \rightarrow 0} \; 0 .
\end{multline*}

\begin{lemma}\label{EW23}
Suppose $\mu,\nu\in K$ and set $E:=|\mu(\Gamma_{out})-\nu(\Gamma_{out})|$. Then we have $E\leq CW_2^{2/3}(\mu,\nu)$, where the constant $C$ depends on the geometry of $\Omega$ and $\Gamma_{out}$.
\end{lemma}
\begin{proof}
Suppose for simplicity $\nu(\Gamma_{out})\geq \mu(\Gamma_{out})$. Take an optimal transport plan $\gamma$ from $\mu$ to $\nu$. Consider $\gamma'=\ONE_{\Oa\times\Gamma_{out}}\gamma$. The mass of $\gamma'$ is a number $E'$, larger than $E$. Let $\mu'$ be the projection of $\gamma'$ on the first variable ($x$): it is a measure with mass $E'$, dominated by $\ONE_{\mathring{\Omega}}\mu$ (and hence it is absolutely continuous with density smaller than $1$). We have
$$W_2^2(\mu,\nu)=\int |x-y|^2\,d\gamma\geq \int |x-y|^2\,d\gamma'\geq\int d(x,\Gamma_{out})^2\,d\gamma'=\int d(x,\Gamma_{out})^2d\mu'.$$
It is sufficient to prove that this last integral is larger than $c(E')^3$. Set $d(x):= d(x,\Gamma_{out})$: we will use the fact that $|[d\leq t ]|\leq ct$. We have
\begin{multline*}
\int d(x)^2d\mu'=\int_0^\infty \mu'\big([d^2>t])dt=\int_0^\infty \big(E'-\mu'([d\leq\sqrt{t}])\big)dt\\
\geq \int_0^\infty \big(E'-|[d\leq\sqrt{t}]|\big)_+dt\geq \int_0^{(E'/c)^2} \big(E'-c\sqrt{t}\big)_+dt=c(E')^3.
\end{multline*}
\end{proof}

At step 2, we prove with the same method that $\mathbf{E}$ is absolutely continuous with respect to the density $\bar \rho:= \mathop{\lim}\limits_{\tau \rightarrow 0} (\rho_{\tau})_{\Omega} = \rho_{\Omega}$ (the decomposition of the measures into a part on $\Gamma_{out}$ and a part on $\Oa$ passes to the limit, because of the density bound on $\Oa$): there exists $\mathbf{u}$ such that $\mathbf{E} = \rho_{\Omega} \mathbf{u}$. Moreover $(\rho, \mathbf{E})$ satisfies the continuity equation, and we can prove again the equality
$$\displaystyle \int_{\Omega} \nabla q \cdot \mathbf{u} \, dx \; = \; 0 \quad \forall \, q \in H_{\rho}^1 .$$

At step 3, the first estimates are still true, since we integrate over $\mathring{\Omega}$ ($p_{\tau} = 0$ on $\Gamma_{out}$). However, we can't use lemma \ref{ineq_geod} anymore. Instead, we apply lemma \ref{ineq_length} and get the inequality
$$\displaystyle \int_{a}^b \int_{\Omega} p_\tau(t,x) \big(\rho_\tau(a,x)-\rho_\tau(t,x)\big) \, dx \, dt \; \leq \; \displaystyle  \int_{a}^b ||\nabla p_\tau(t,.)||_{L^2(\Omega)} L(\rho_{\tau}(a,.) , \rho_{\tau}(t,.)) \, dt .$$
Using proposition (\ref{contL}) and the same notation as in Section~\ref{sec:proofnoexit}, step 3, the limit $\tau \rightarrow 0$ reads:
$$
\displaystyle \mathop{\lim}\limits_{\tau \rightarrow 0} \dfrac{1}{b-a} \int_{a}^b \int_{\Omega} p_\tau(t,x)\big(\rho_\tau(a,x)-\rho_\tau(t,x)\big) \, dx \, dt 
$$
$$  \leq  \displaystyle \dfrac{1}{b-a} \sqrt{\mu([a,b])} \left( \int_{a}^b L(\rho(a), \rho(t))^2 \, dt \right)^{\frac{1}{2}} .
 $$ 
Since at the limit, the curve $\rho(t)$ belongs to $K$ for every t, we have also:
 $$\displaystyle L(\rho(a), \rho(t)) \, dt \leq \int_{a}^t |\rho'|_{W_2}(s) \, ds \; \leq \; \left( \int_{a}^t |\rho'|^2_{W_2}(s)\,ds\right)^{1/2}(t-a)^{1/2}\leq C(b-a)^{1/2}.$$ 
 Therefore, we have the following inequality:
$$\begin{array}{rcl}
\displaystyle \mathop{\lim}\limits_{\tau \rightarrow 0} \dfrac{1}{b-a} \int_{a}^b \int_{\Omega} p_\tau(t,x) \big(\rho_\tau(a,x)-\rho_\tau(t,x)\big) \, dx \, dt & \leq & \displaystyle  \dfrac{1}{b-a} \sqrt{\mu([a,b])} \left( \int_{a}^b C (b-a) \, dt \right)^{\frac{1}{2}} \\
& = &   C\sqrt{ \mu([a,b])} \; \mathop{\longrightarrow}\limits_{b \rightarrow a} \; 0 \; \textmd{ for a.e. } a,
\end{array}$$ 
and we conclude the proof as in the particular case $\Gamma_{out} = \emptyset$.

\section{Illustration: a convergent corridor}

We present here an example where both the transport equation and discrete process of the gradient-flow problem can be solved quasi-explicitely. We also give numerical estimations on the convergence of the discrete scheme to the solution of the continuity equation.

We want to model the displacement of a crowd throught a convergent corridor. We thus take for $\Omega$ a portion of a cone, expressed in polar coordinates as $\big[r\in [a,R],\,\theta\in [-\alpha,\alpha]\big]$ (see fig \ref{cone}), with a possible ``exit'' $\Gamma_{out} = \{a\} \times [-\alpha, \alpha]$, and we take for $D$ the distance to the exit (or to the apex, which is equivalent): $D(r) = r$. We assume that the initial density is uniform: $\rho^0(r) = \rho^0 < 1$. We will consider in this section two examples: the case $a=0$ with no exit (so that people will in the end concentrate on the neighborhood
 of the vertex) and the case $a>0$ with exit.

\vspace{0.5cm}
\begin{figure}[!h]
\begin{center}
\psfrag{a}[l]{$a$}
\psfrag{U}[l]{$U$}
\psfrag{R}[l]{$R$}
\psfrag{b(t)}[l]{$b(t)$}
\psfrag{OO}[l]{$\OO$}
\psfrag{GGout}[l]{$\GG_{out}$}
\psfrag{GGw}[l]{$\GG_{w}$}
\includegraphics[width=0.5\linewidth]{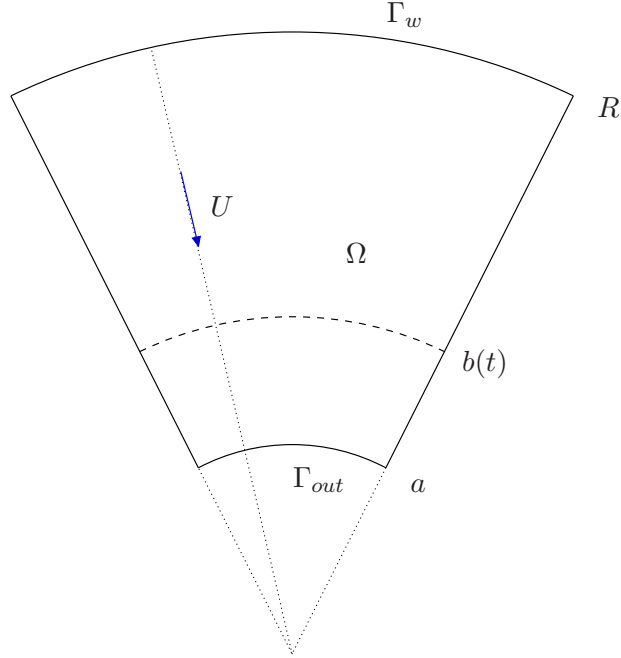}
\caption{Modeling of the displacement of a crowd throught a convergent corridor.}
\label{cone}
\end{center}
\end{figure}

Thanks to the symmetry of this problem, the minimizing movement scheme can be written as a minimization problem on the transport function:
$$
\left\{ \begin{array}{rcl}
\rho_\tau^0 & = & \rho^0\\
\rho_\tau^{k} & = & \mathbf{s_{k}}_{\#} \rho_\tau^{k-1} ,\; \mathbf{s_{k}} \in \; \mathop{\textmd{argmin}}\limits_{\mathbf{t}_{\#} \rho_\tau^{k-1} \in K} \left\{ \displaystyle \int_{a}^R \left( D(\mathbf{t}(r)) + \dfrac{1}{2 \tau} |r-\mathbf{t}(r)|^2 \right) \rho_\tau^{k-1}(r) \, r \, dr \right\}.
\end{array} \right.
$$

Let us first consider the case where $a=0$ (and $\Gamma_{out} = \emptyset$), where this problem can be explicitely solved: $\rho_{\tau}^k$ is given by
\begin{equation}
\rho_\tau^k (r) = \left\{ \begin{array}{lll}
1 & \textmd{on} & [a, b_{\tau}^k[\\
\rho^{0} \left( 1 + \dfrac{k \tau}{r} \right) & \textmd{on} & [b_{\tau}^k, R-k \tau[\\
0 & \textmd{on} & [R-k\tau, R],
\end{array} \right.
\end{equation}
where $b_{\tau}^k$ satisfies the recurrence relation
\begin{equation}
\left\{ \begin{array}{l}
b_{\tau}^0 = 0\\
(b_{\tau}^k)^2 - \rho^0 (b_{\tau}^k + k \tau)^2 \; =  \; (b_{\tau}^{k-1})^2 - \rho^0 (b_{\tau}^{k-1} + (k-1) \tau)^2,
\end{array} \right.
\end{equation}
and the solution of the continuity equation can be easily calculated:
$$
\rho(t,r) = \left\lbrace 
\begin{array}{ll}
1 & \textmd{if} \; r \in \; [a, b(t) [\\
\rho^0  \left( 1 + \dfrac{t}{r} \right) & \textmd{if} \; r \in \; [b(t), R - t]\\
0 & \textmd{if} \; r \in \;  [R- t , R],\\
\end{array}
\right. 
\; \textmd{where} \;
\left\{ \begin{array}{lll}
b(0) & = & 0\\
b'(t) & = & \rho^{0} \; \dfrac{b(t) +  t}{b(t) - \rho^{0} \left( b(t) +t \right)} .
\end{array} \right.
$$
In figure \ref{no_exit}, we represent the discrete densities $\rho_{\tau}^k$ at different times for the numerical values $\tau = 0.01$, $a=0, R=10$, and $\rho^0=0.4$. Let us remark that the recurrence relation that satisfies $b_{\tau}^k$ is a numerical scheme for the ODE on $b(t)$. Indeed, it writes
$$\dfrac{b_{\tau}^k - b_{\tau}^{k-1}}{\tau} = F \left( \dfrac{b_{\tau}^k + b_{\tau}^{k-1}}{2} , \dfrac{2k-1}{2} \tau \right), \quad \textmd{where} \quad F(r,t) = \rho^{0} \; \dfrac{b(t) +  t}{b(t) - \rho^{0} \left( b(t) +t \right)}.$$
Using the conservation of the total amount of people, it is easy to prove that this scheme is exact at every time step $k \tau$, and so is the discrete solution $\rho_{\tau}^k$.

\vspace{0.5cm}
\begin{figure}[!h]
\begin{center}
\hspace*{-0.5cm}  \begin{minipage}[c]{0.25\linewidth}
 \includegraphics[width=\linewidth]{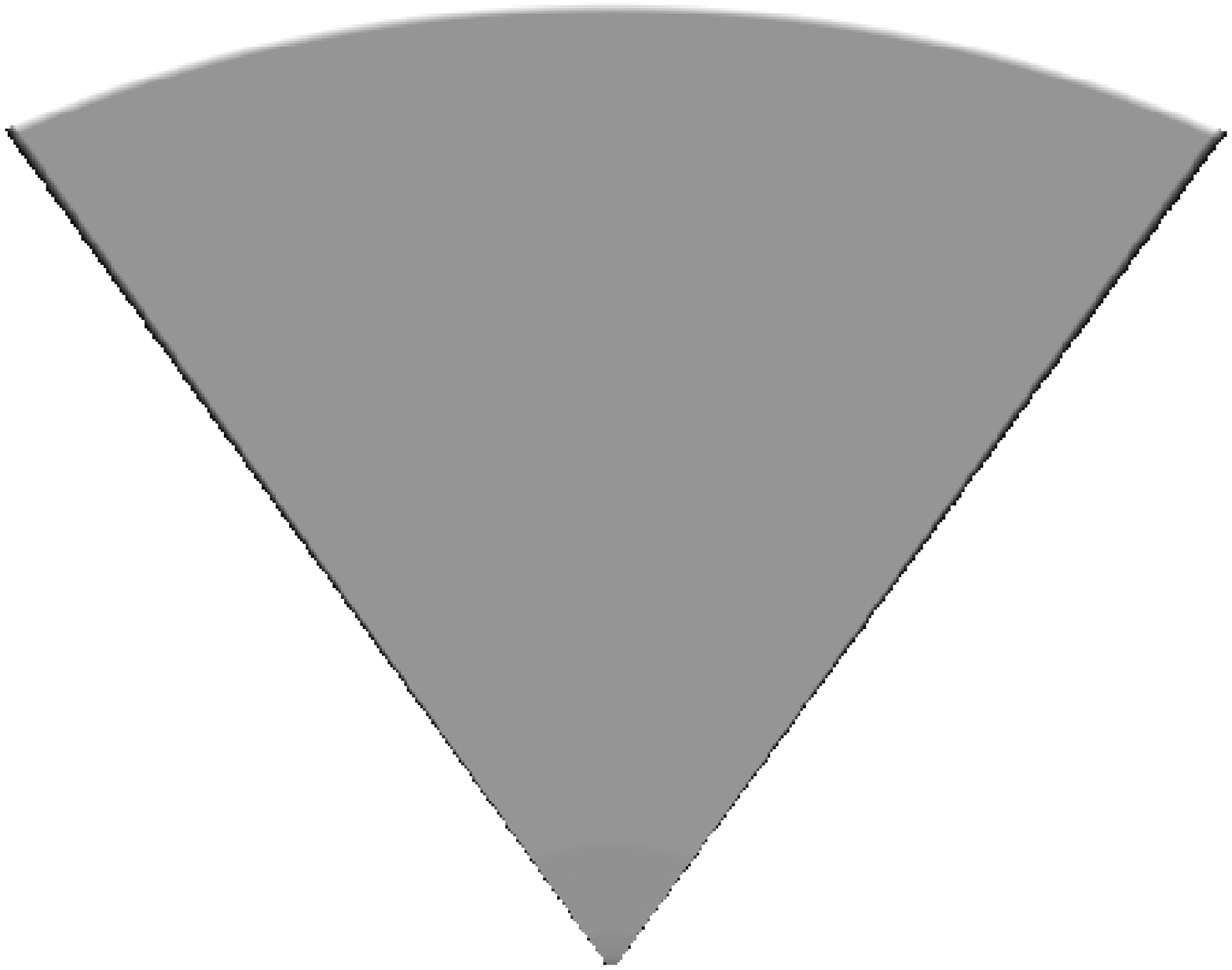}
 \end{minipage}
\begin{minipage}[c]{0.25\linewidth}
  \includegraphics[width=\linewidth]{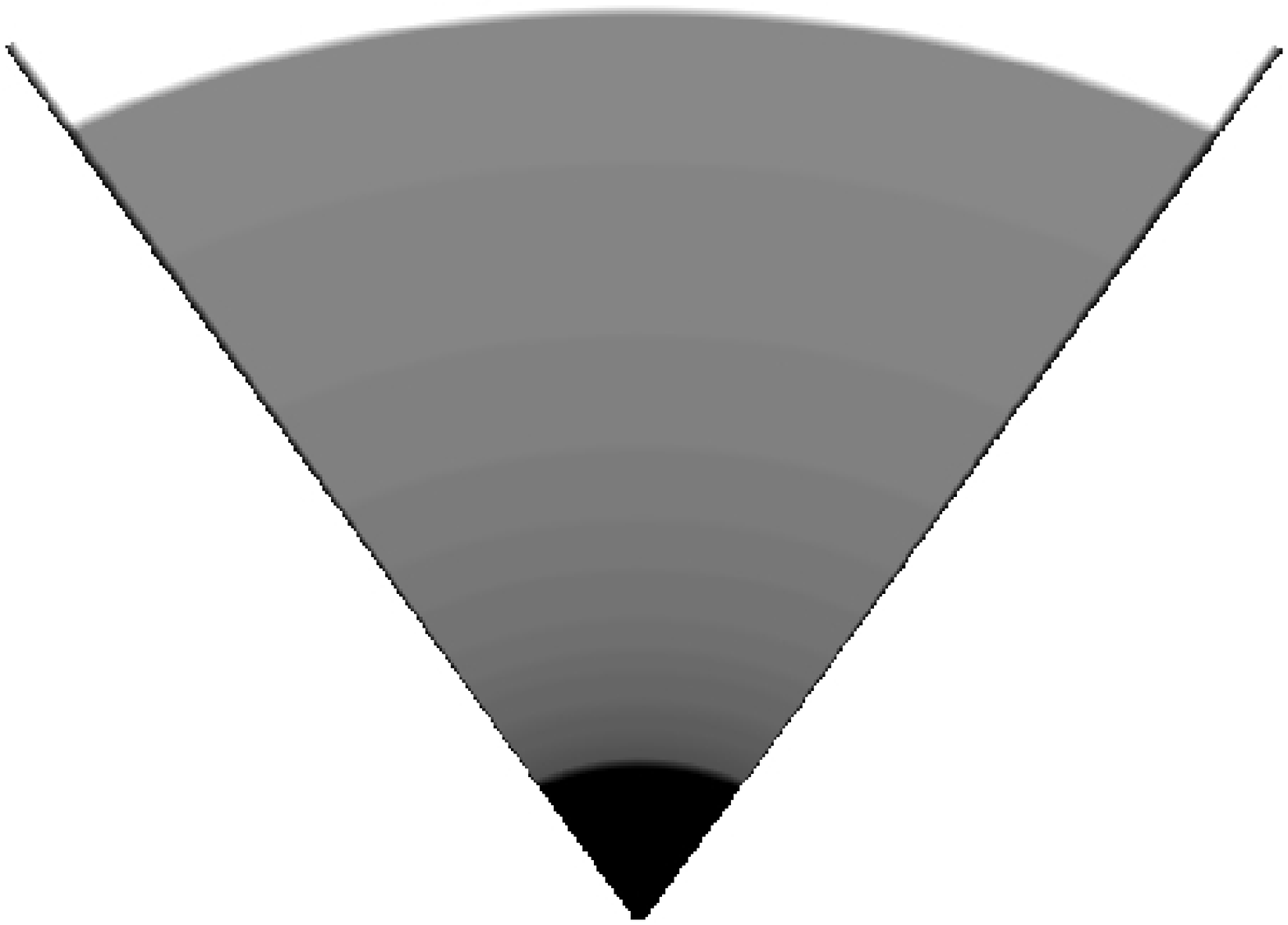}
 \end{minipage}
 \begin{minipage}[c]{0.25\linewidth}
  \includegraphics[width=\linewidth]{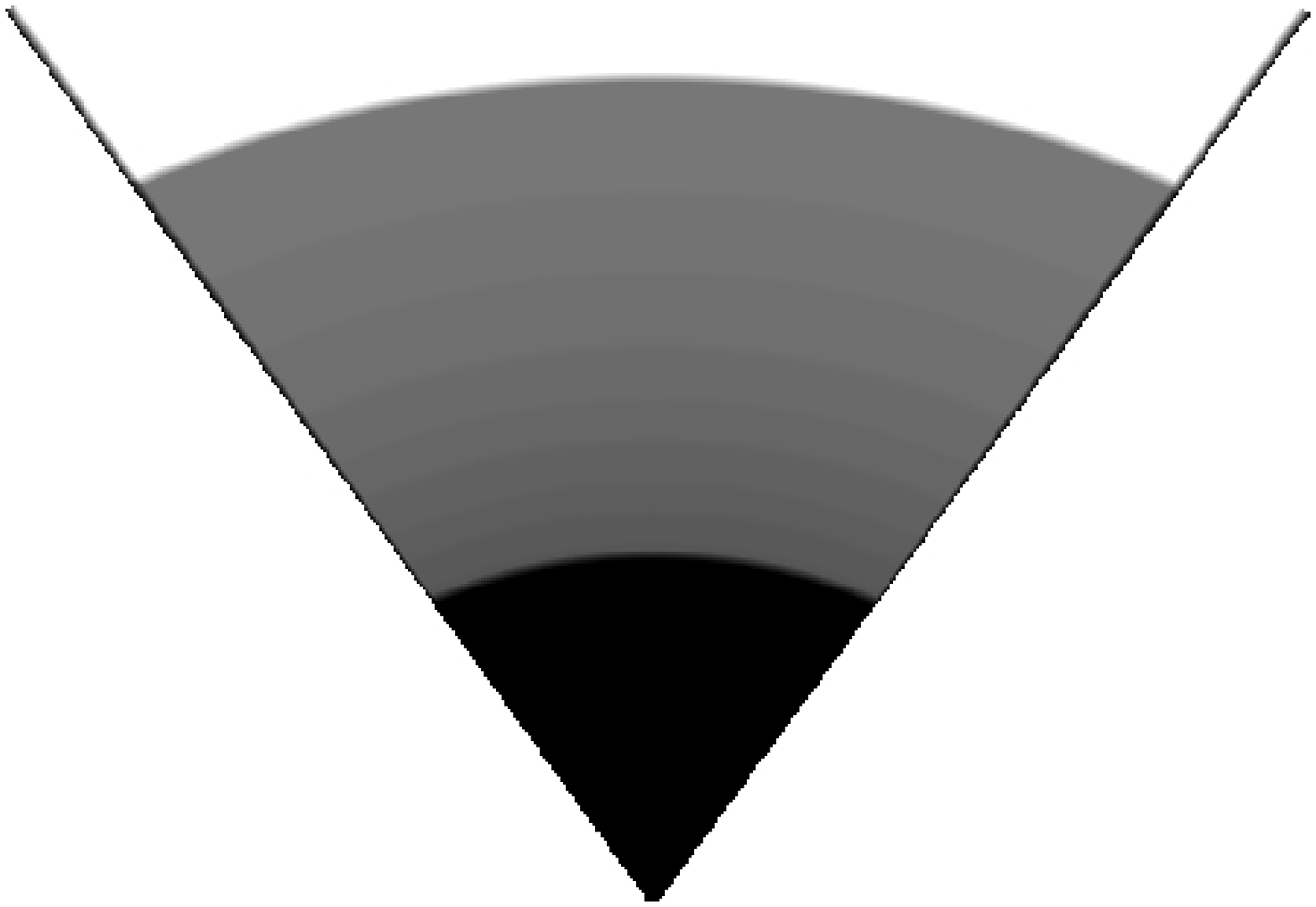}
 \end{minipage}
\begin{minipage}[c]{0.25\linewidth}
  \includegraphics[width=\linewidth]{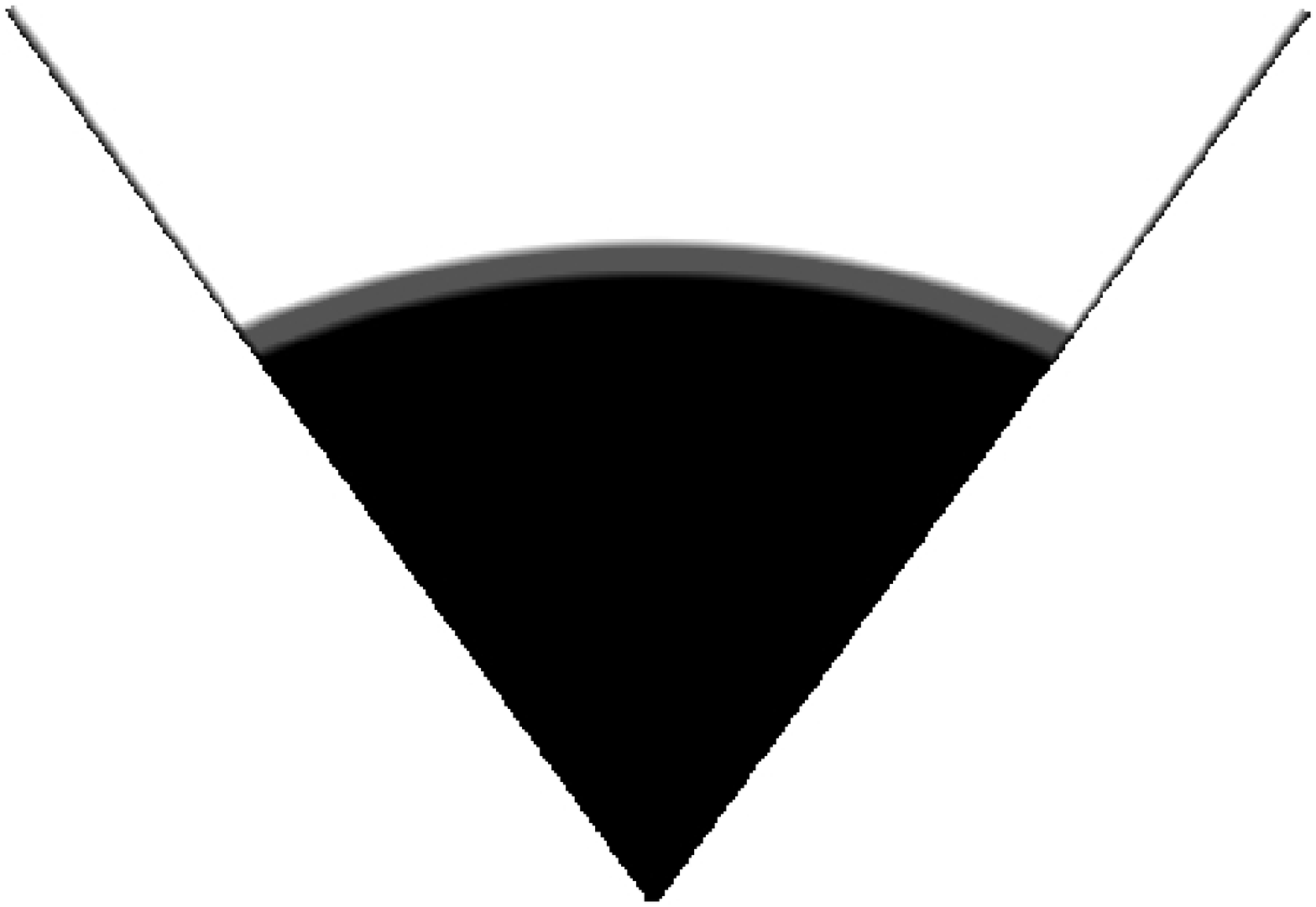}
 \end{minipage}
\caption{Evolution of the solution of the minimizing movement scheme in the case where $\Gamma_{out} = \emptyset$.}
\label{no_exit}
\end{center}
\end{figure}

We now consider the case $a>0$ with exit. The densities have then the same form, except that the evolution of the interface in the continuous case is now given by
$$\left\{ \begin{array}{rcl}
b'(t) & = & \Phi(t,b(t))\\
b(t_{0}) & = & a,
\end{array} \right. \quad \textmd{with} : \; \Phi(t,r) = \left\{ \begin{array}{ll} 
\dfrac{ \rho_{0} \left( 1 + \dfrac{t}{r} \right)  - \dfrac{r-a}{r \, \ln ( r/a)}}{ 1 -  \rho_0  \left( 1 + \dfrac{t}{r} \right)  } &\textmd{if} \; r \leq R - t \\
-  \dfrac{r-a}{r \, \ln ( r/a)} & \textmd{if} \; r > R - t ,
\end{array} \right.$$
whereas in the discrete case, $b_{\tau}^k$ satisfies now the recurrence relation
$$
(b_{\tau}^k)^2 - a^2 - \rho_{0} (b_{\tau}^k + k \tau)^2  =  (b_{\tau}^{k-1})^2 - r_{e}^2 - \rho_{0} (b_{\tau}^{k-1} + (k-1) \tau)^2 
$$
if $ b_{\tau}^{k-1} < R - (k-1)\tau$, and 
$$
(b_{\tau}^k)^2 - a^2  =  (b_{\tau}^{k-1})^2 - r_{e}^2 $$
if $ b_{\tau}^{k-1} \geq R - (k-1)\tau$,
where $r_{e}$ is the (unknown) radius such that people who were between $a$ and $r_e$ at step $k-1$ will exit the corridor (i.e. arrive at $a$) at step $k$. This radius is given as the minimum of an integral expression that we will not develop here.
In figure \ref{exit}, we represent the discrete densities for $a=1$ and for the same numerical values as before.

\vspace{0.5cm}
\begin{figure}[!h]
\begin{center}
\hspace*{-0.5cm}  \begin{minipage}[c]{0.25\linewidth}
 \includegraphics[width=\linewidth]{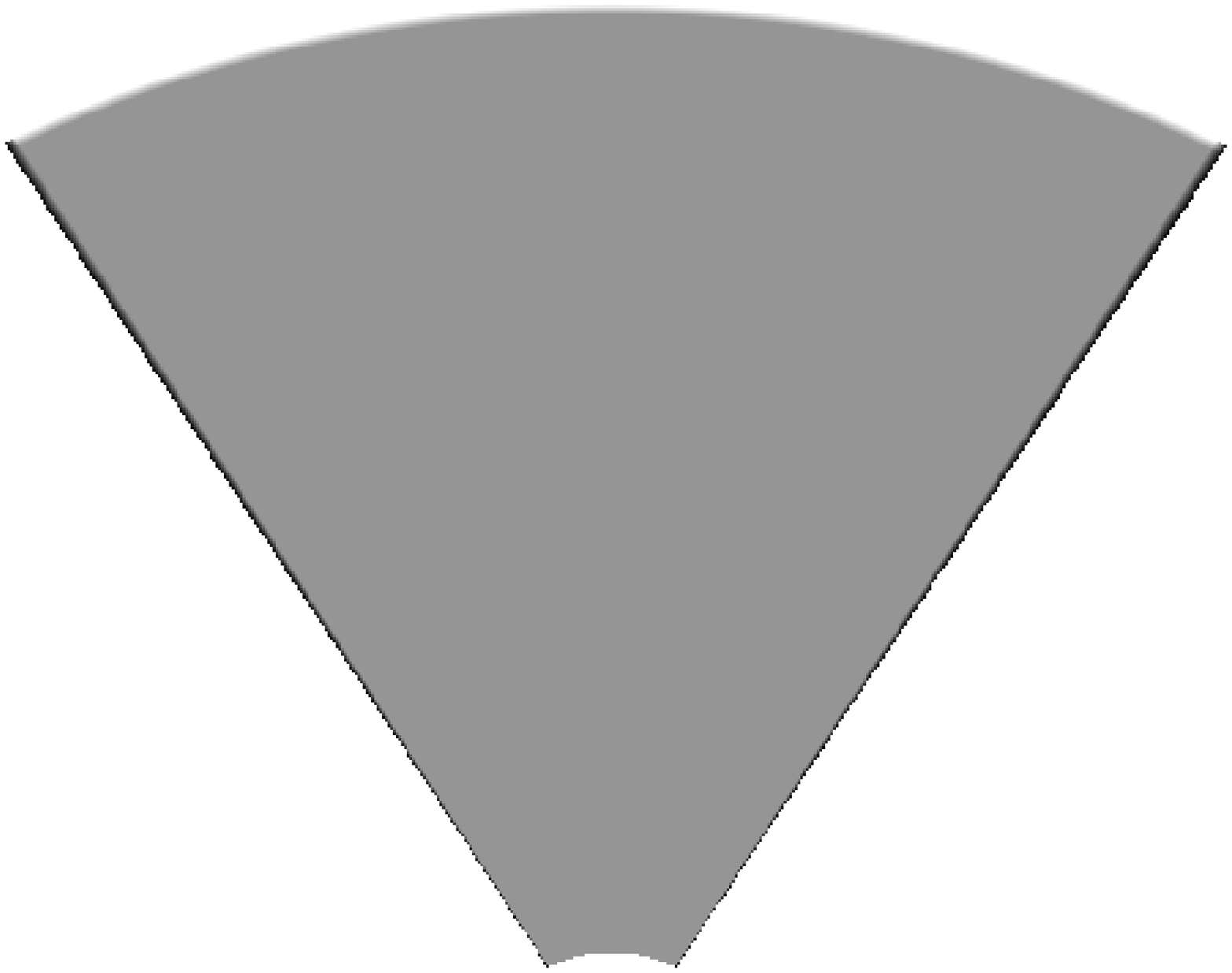}
 \end{minipage}
\begin{minipage}[c]{0.25\linewidth}
 \includegraphics[width=\linewidth]{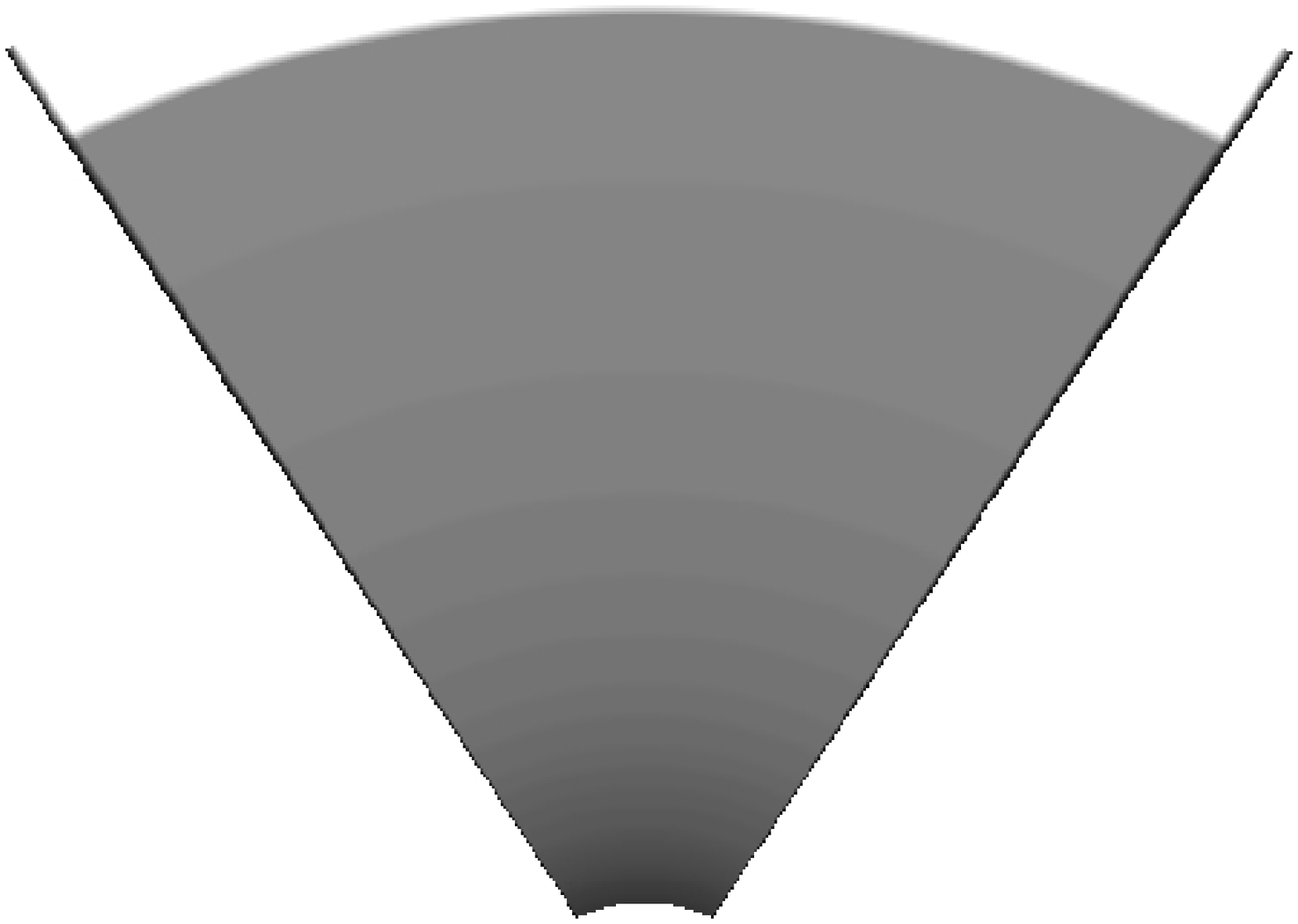}
 \end{minipage}
 \begin{minipage}[c]{0.25\linewidth}
  \includegraphics[width=\linewidth]{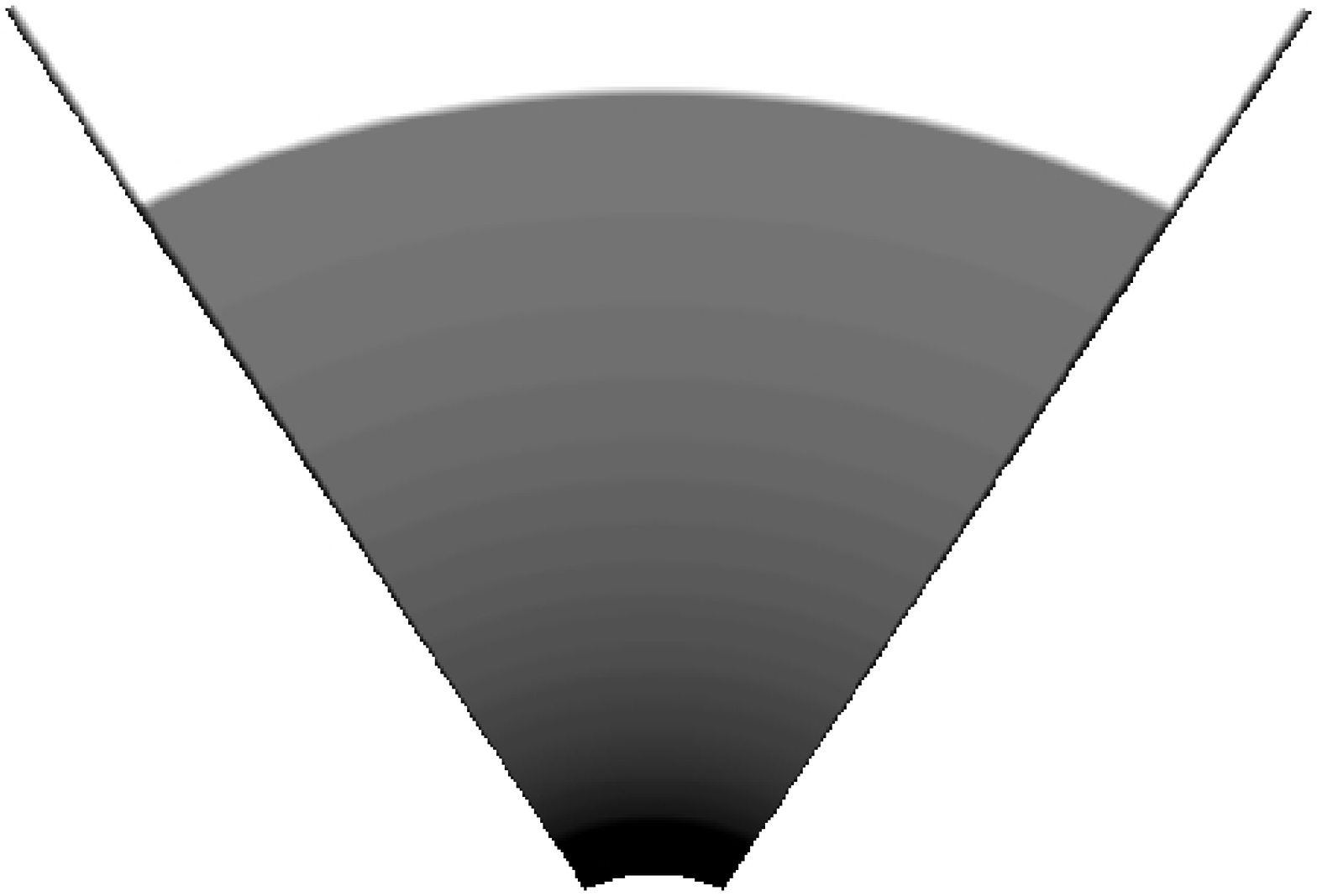}
 \end{minipage}
\begin{minipage}[c]{0.25\linewidth}
  \includegraphics[width=\linewidth]{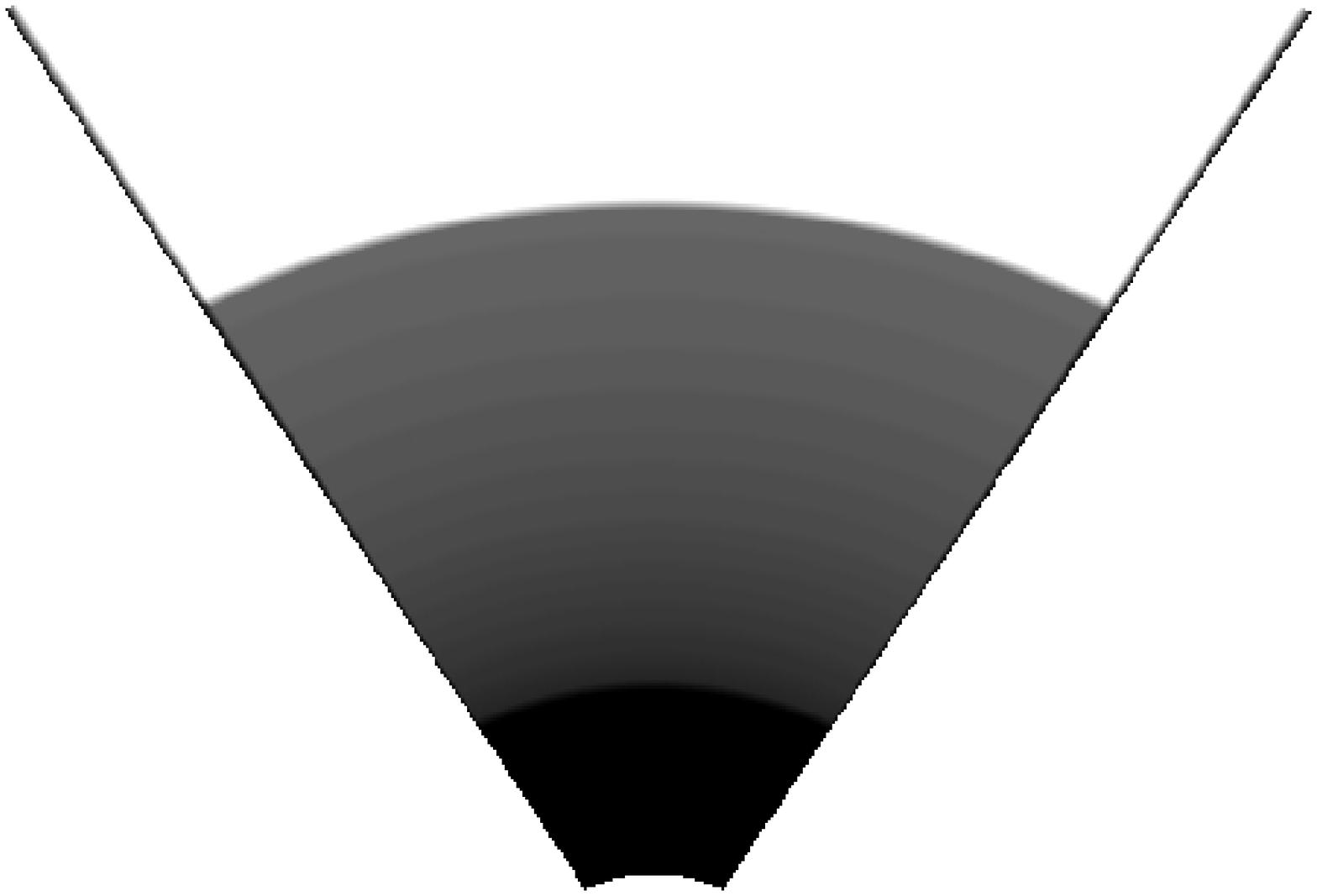}
 \end{minipage}\\
 \hspace*{-0.5cm}  \begin{minipage}[c]{0.25\linewidth}
 \includegraphics[width=\linewidth]{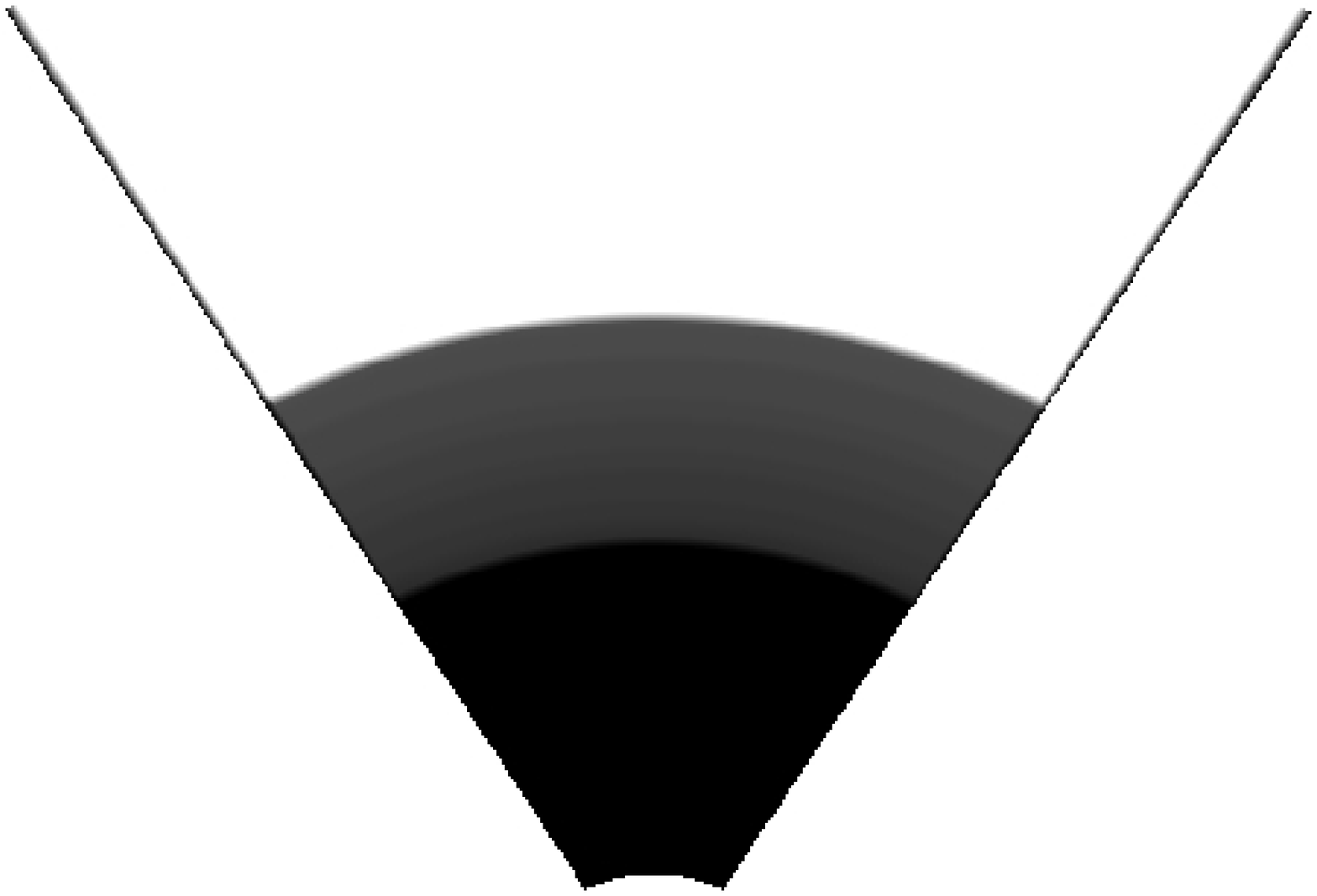}
 \end{minipage}
\begin{minipage}[c]{0.25\linewidth}
  \includegraphics[width=\linewidth]{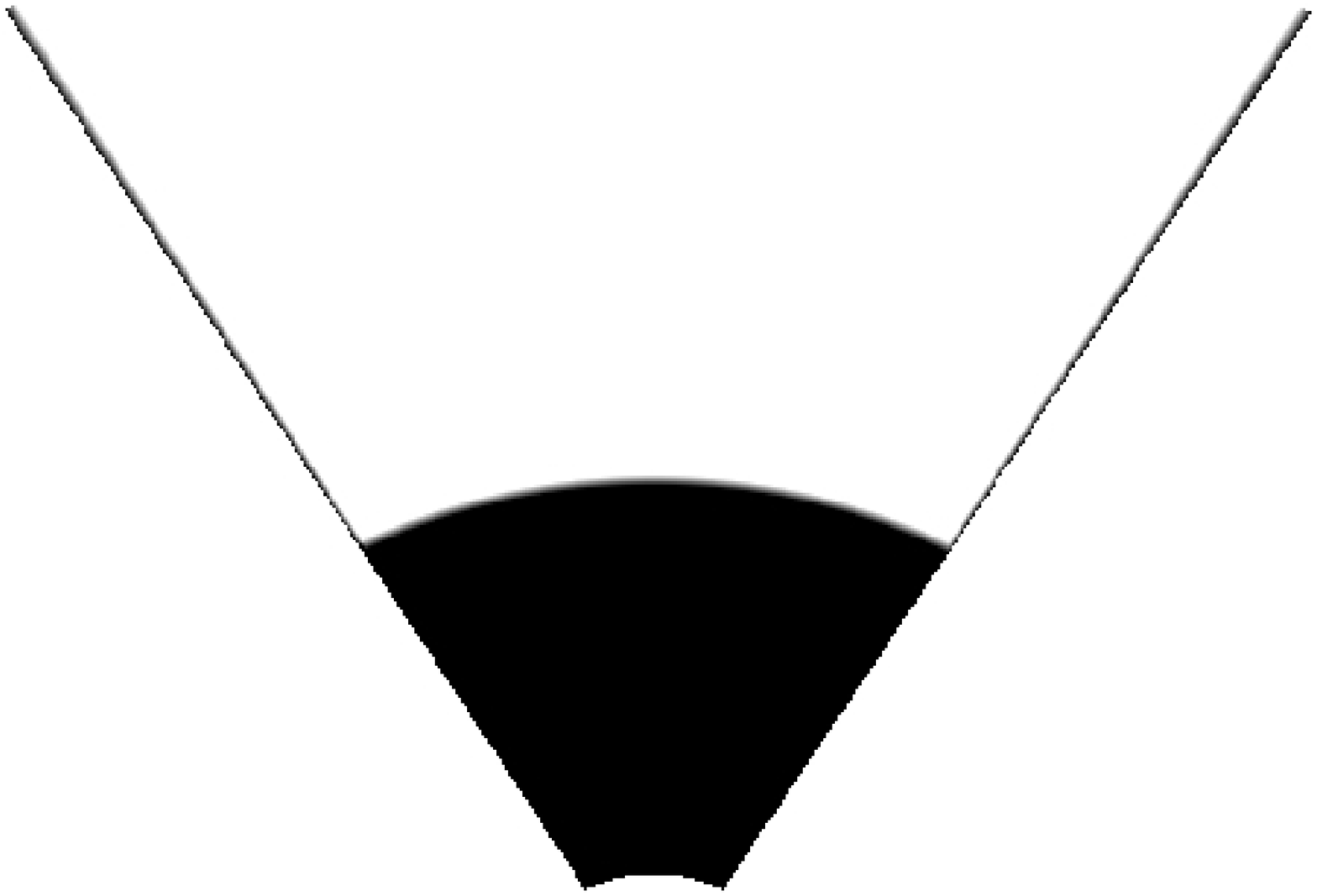}
 \end{minipage}
 \begin{minipage}[c]{0.25\linewidth}
  \includegraphics[width=\linewidth]{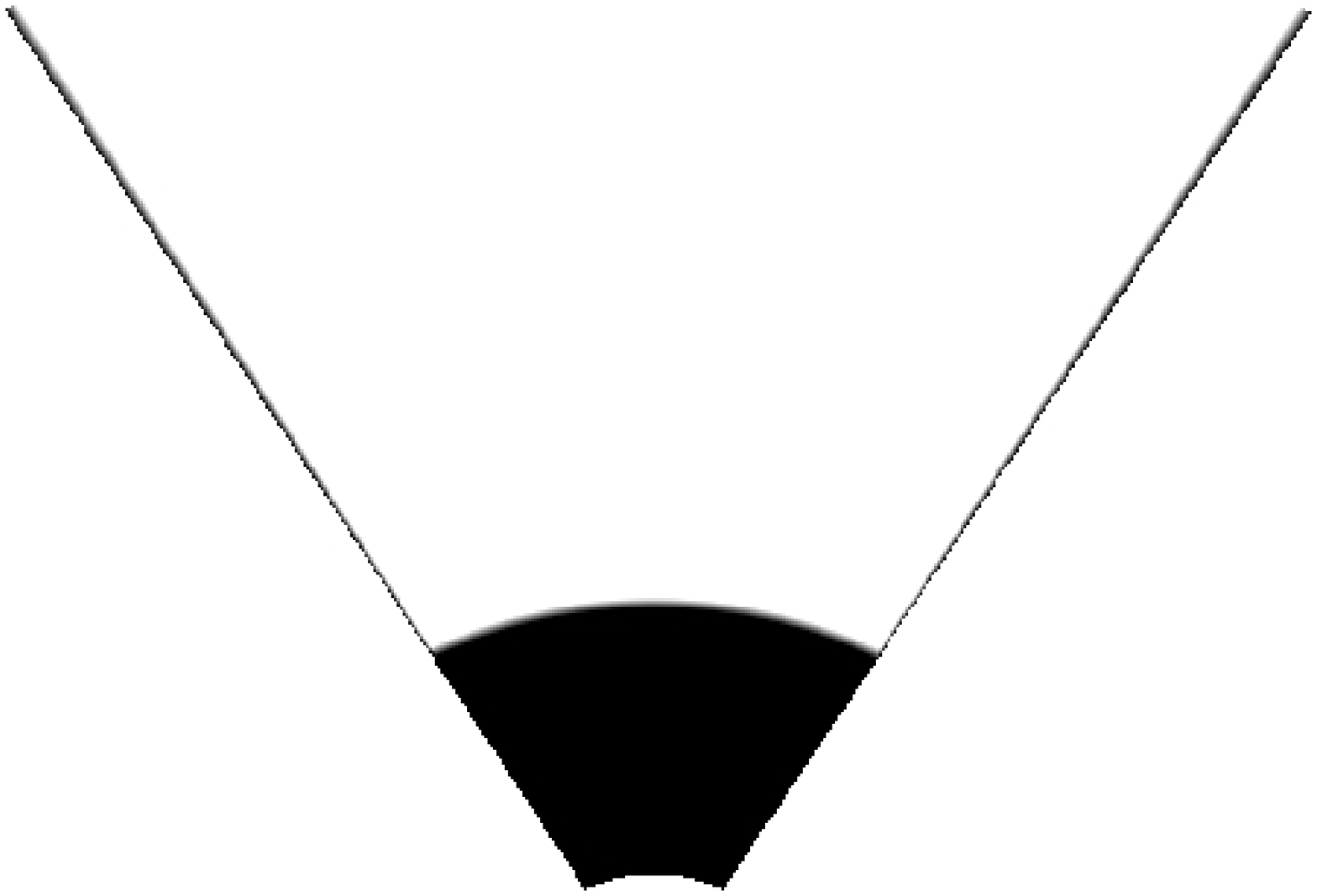}
 \end{minipage}
 \begin{minipage}[c]{0.25\linewidth}
  \includegraphics[width=\linewidth]{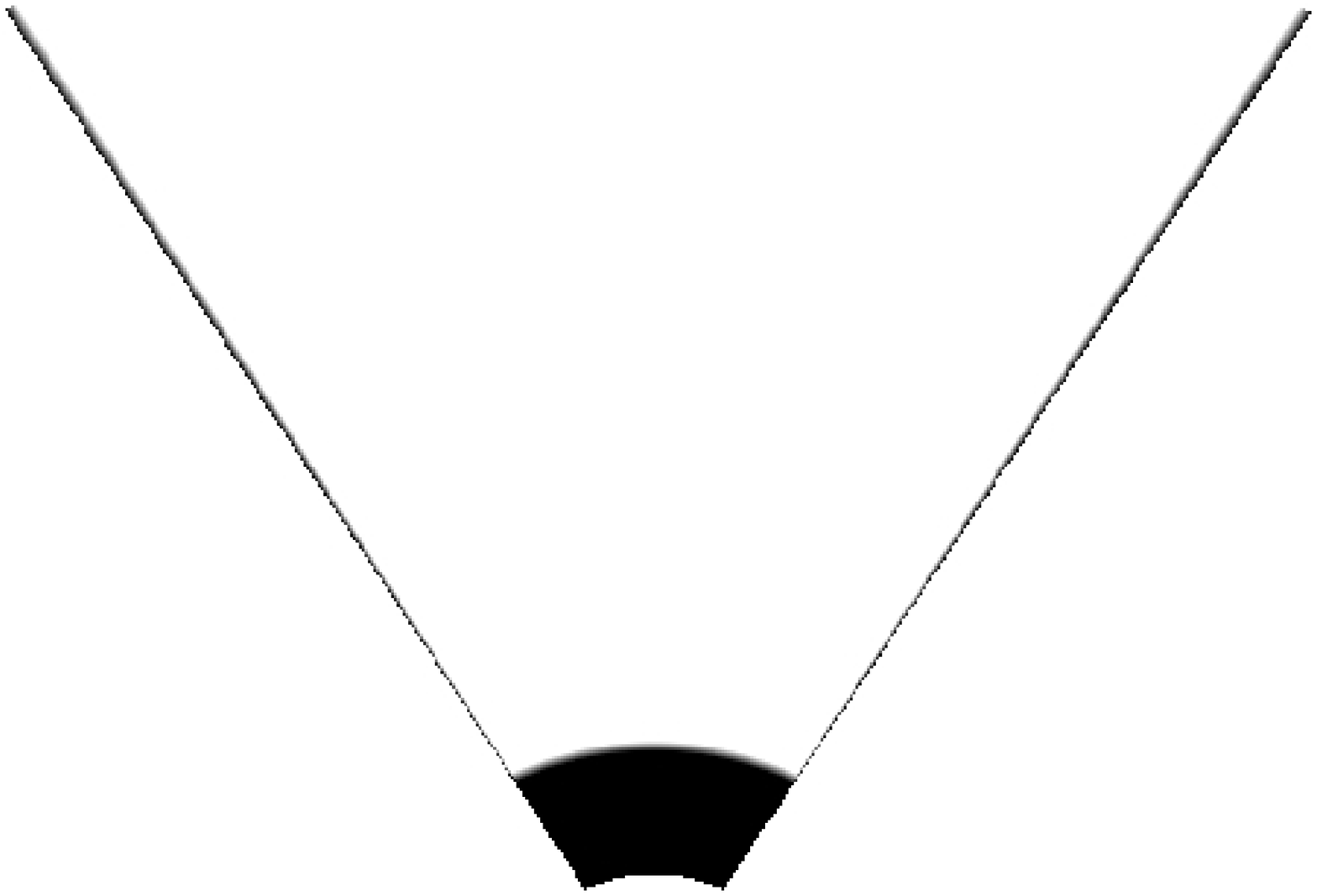}
 \end{minipage}
\caption{Evolution of the solution of the minimizing movement scheme in the case with exit.}
\label{exit}
\end{center}
\end{figure}

It is also interesting to estimate numerically the error between the solution of the continuity equation and the solution of the minimizing movement scheme. In this purpose, we consider the case where the density is initially saturated ($\rho^0=1$), and we compute $b$ and $b_{\tau}$ with high accuracy (high order method for the ODE on $b$, and precise quadrature and optimization methods to estimate $r_{e}$ and $b_{\tau}$), so that space discretization does not affect error estimation. We obtain numerically that $b_{\tau}$ converges to $b$ when $\tau$ tends to $0$ with an error of order 1 (interpolation polynom of $\tau \mapsto |b(T)-b_{\tau}(T)|$ gives order 0.989 for $T=1$), which gives also an order 1 error for the Wasserstein distance between $\rho$ and $\rho_{\tau}$.

\section{Modelling issues, extensions}

We would like to conclude this paper by some remarks on the limitations of the overall approach in terms of modelling, and on possible extensions to other domains. 

With its very macroscopic and eulerian nature, the model is designed to handle  large populations as a whole, 
and does not allow to localize and follow in their path individual pedestrians.
 As a direct consequence, the spontaneous (or desired) velocity of an individual may depend on its location only, so that differentiated individual strategies (e.g. avoidance of crowded zone, skirting of obstacles) cannot be included straightforwardly.
Besides, the macroscopic expression of the non-overlapping constraint is less restrictive than its microcopic counterpart. In the  microscopic setting (people are identified with rigid discs), for highly packed situations, the non overlapping constraints induce some kind of non-negative divergence constraint in the directions of contacts. Consider the example of a cartesian distribution of monodisperse discs  (see Fig.~\ref{fig:os}, left), with a uniform  spontaneous velocity  directed toward a wall. The actual velocity will be $0$, whereas in the macroscopic version it is not (see Fig.~\ref{fig:os}, center). Note that, if the microscopic distribution is modified, while mean density is preserved, the situation is no longer static (see Fig.~\ref{fig:os}, right). 
This example illustrates the deep difference between the microscopic approach (for which local structure and orientation of contacts lines play an essential role), and the macroscopic one, which only considers local density.
\begin{figure}[!h]
\begin{center}

\includegraphics[width=0.8\textwidth]{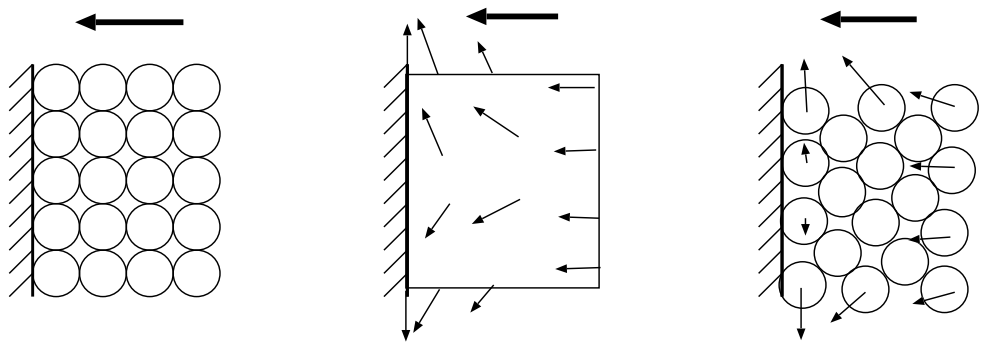}
\caption{Differences between micro and macro approaches}
\label{fig:os}
\end{center}
\end{figure}

As a consequence the model is unlikely to reproduce  the formation
of blocked archs near an exit, which are observed in practice in highly critical emergency situations, and which can be recovered by the microscopic model, even without friction. Note also that, together with the structural anisotropy we just mentionned,  individual  anisotropy of pedestrians is not handled (whereas it may be in microscopic models by replacing discs by ellipses, for example). 

Yet, despite these limitations in the modelling of crowd motion, we believe that this new type of evolution problem may be 
fruitful to model phenomena, in particular in the domain of cell dynamics. In this context, the spontaneous velocity would be replaced by some kind of chemotaxis velocity. As an illustration, let us express what could be called a unilateral version of Keller-Segel equation, in the spirit of what has been presented here:
$$
\frac{\partial \rho}{\partial t} + \nabla \cdot(\rho \uu) = 0
$$
$$
\uu= P_{C_\rho} \UU\virg \UU = \nabla c\virg 
-\Delta c = \rho,
$$
where $c$ denotes the concentration of some attracting agent, generated by the cells themselves.
Notice that the congestion constraint prevents concentration of mass.
As a matter of fact,  the characteristic function of a single ball is a static solution to this system, in the whole space $\setR^3$, and it can be expected that any solution converges to such a configuration.
Note also that bacterial growth could be handled  by adding an appropriate term in the right-hand side of the transport equation.

We also believe that it can be fruitful in the modelling of granular media.
Bouchut et al.\cite{bouchut} propose a model of pressureless gas for which the
density is subject to remain less than 1. This model is essentially mono-dimensionnal
(the construction of explicit solutions proposed in \cite{berth} uses extensively the one-dimensionality).
As our model can be seen as a first order (in time) version of this
second order pressureless gas model, we believe that the handling of the congestion
constraint we propose here, which applies in any dimension, might be used in the
future for a macroscopic description of granular flows in higher dimension. 
The corresponding model, which is a second order version of the model we considered here, could be written as follows
$$
\frac {\partial \rho}{\partial t} + \nabla \cdot(\rho \uu) = 0,
$$
$$
\frac {\partial( \rho\uu)}{\partial t} + \nabla \cdot(\rho\uu\otimes \uu)  + \nabla p= 0,
$$
$$
0\leq \rho \leq 1\virg p(1-\rho) = 0\virg \uu^{+}= P_{C_\rho} \uu^{-}\virg
$$
where $C_\rho$ is the cone of feasible velocities which we introduced, and $\uu^-$ (resp. $\uu^+$) is the velocity before (resp. after) the collision (both are equal in case there is no collision).

Those extensions are not straightforward. In  particular,  the   ``spontaneous'' velocity (i.e. chemotactic velocity for cell models, velocity before collision for granular flow models) may not be a  gradient, and may furthermore vary in time.
It  rules out some arguments we used here. The discrete gradient flow construction, which is based on a simultaneous handling of advection and congestion constraint,  could be replaced by a  prediction-correction strategy: a first step of free advection  followed by a projection (for the Wasserstein distance) onto the set of feasible densities. This approach would reproduce in the Wasserstein setting the techniques which proved successful in the context  of sweeping processes in Hilbert spaces, which we mentioned in the introduction  (see Refs.~\cite{Mor}, \cite{Thi1}, \cite{Thi2}).


\section*{Acknowledgment}
We would like to thank Y. Brenier, A. Figalli and  N. Gigli for  fruitful discussions.






\begin{thebibliography}{1}

   \bibitem{MovMin} L. Ambrosio, Movimenti minimizzanti,  {\it Rend. Accad. Naz. Sci. XL Mem. Mat. Sci. Fis. Natur.} \textbf{113} (1995) 191--246.
 
   \bibitem{Amb} L. Ambrosio, N. Gigli, G. Savar\'{e}, Gradient flows in metric spaces and in the space of probability measures, {\it Lectures in Mathematics} (ETH Z\"{u}rich, 2005).
   
    \bibitem{AmbSav} L. Ambrosio, G. Savar\'{e}, Gradient flows of probability measures, {\it Handbook of differential equations, Evolutionary equations} \textbf{3}, ed. by C.M. Dafermos and E. Feireisl (Elsevier, 2007).
    
    \bibitem{Bel} N. Bellomo, C. Dogbe, On the modelling crowd dynamics from scalling to hyperbolic macroscopic models, {\it Math. Mod. Meth. Appl. Sci.} \textbf{18} Suppl. (2008) 1317--1345.
   
   \bibitem{BenBre} J.-D. Benamou, Y. Brenier, A computational fluid mechanics solution to the Monge-Kantorovich mass transfer problem, {\it Numer. Math.} \textbf{84}(3) (2001) 375--393.
   
   \bibitem{berth} F. Berthelin, Existence and weak stability for a pressureless model with unilateral constraint, 
{\it Math. Mod. Meth. Appl. Sci.} \textbf{12}(2) (2002) 249--272.
   
   \bibitem{BlaCalCar} A. Blanchet, V. Calvez, J.A. Carrillo, Convergence of the mass-transport steepest descent scheme for the sub-critical Patlak-Keller-Segel model, {\it SIAM J. Numer. Anal.} to appear.
   
   \bibitem{bouchut} F. Bouchut, Y. Brenier, J.  Cortes, J.-F Ripoll,  A hierarchy of models for two-phase flows,
{\it J. nonlinear sci.} \textbf{10}(6) (2000) 639--660.
   
   \bibitem{But1} G. Buttazzo, C. Jimenez, E. Oudet, An optimization problem for mass transportation with congested dynamics, {\it SIAM J. Control Optim.} (2007).
   
   \bibitem{ButSan} G. Buttazzo, F. Santambrogio, A model for the optimal planning of an urban area, {\it SIAM J. Math. Anal.} \textbf{37}(2) (2005) 514--530.
   
   \bibitem{CarGuaTos} J.A. Carrillo, M.P. Gualdani, G. Toscani, Finite speed of propagation in porous media by mass transportation methods, {\it R. Acad. Sci. Paris Ser. I} \textbf{ 338} (2004) 815--818.   
%
%
   
    \bibitem{Cha1} C. Chalons, Numerical approximation of a macroscopic model of pedestrian flows, {\it SIAM J. Sci. Comput.} \textbf{29}(2) (2007) 539--555.

    \bibitem{Cha2} C. Chalons, Transport-equilibrium schemes for pedestrian flows with nonclassical shocks, {\it Traffic and Granular Flows'05}, 347--356 (Springer, 2007).

    \bibitem{Col}  R.M. Colombo, M.D. Rosini, Pedestrian flows and non-classical shocks,  {\it Math. Methods Appl. Sci.} \textbf{28} (2005) 1553--1567.
   
    \bibitem{Cos} V. Coscia, C. Canavesio, First-order macroscopic modelling of human crowd dynamics, {\it Math. Mod. Meth. Appl. Sci.} \textbf{18} (2008) 1217--1247.
 
   \bibitem{DacMos} B. Dacorogna, J. Moser,
On a partial differential equation involving the Jacobian determinant, {\it Annales de l'Institut Henry Poincar\'e Analyse Non Lin\'eaire} \textbf{7}(1) (1990) 1--26.

    \bibitem{DeG} E. De Giorgi, New problems on minimizing movements, {\it Boundary
Value Problems for PDE and Applications, C. Baiocchi and J. L. Lions
eds.} (Masson, 1993) 81--98.

     \bibitem{Deg} P. Degond, L. Navoret, R. Bon, D. Sanchez, Congestion in a macroscopic model of self-driven particles modeling gregariousness, to appear.
     
     \bibitem{Dog} C. Dogbe, On the numerical solutions of second order macroscopic models of pedestrian flows, {\it Comput. Appl. Math.} \textbf{56}7 (2008) 1884--1898.

     \bibitem{Thi1} J.F. Edmond, L. Thibault, Relaxation of an optimal control problem involving a perturbed sweeping process, {\it Math. Program Ser B} \textbf{104}(2-3) (2005) 347--373. 
     
     \bibitem{Thi2} J.F. Edmond, L. Thibault, BV solutions of nonconvex sweeping process differential inclusion with perturbation, {\it J. Differential Equations} \textbf{226}(1) (2006) 135--179.

     \bibitem{Helb1} D. Helbing, A fluid dynamic model for the movement of pedestrians, {\it Complex Systems} \textbf{6} (1992) 391--415.

    \bibitem{Helb2} D. Helbing, P. Molnar, F. Schweitzer, Computer simulations of pedestrian dynamics and trail formation, {\it Evolution of Natural Structures, Sonderforschungsbereich} \textbf{230} (Stuttgart, 1994) 229--234.
    
    \bibitem{Helb3} D. Helbing, P. Moln\'{a}r, Social force model for pedestrian dynamics, {\it Phys. Rev E} \textbf{51} (1995) 4282--4286.
 
    \bibitem{Helb4} D. Helbing, T. Vicsek, Optimal self-organization, {\it New J. Phys.} \textbf{1} (1999) 13.1--13.17.
      
    \bibitem{Helb5} D. Helbing, I. Farkas, T. Vicsek, Simulating dynamical features of escape panic, {\it Nature} \textbf{407} (2000) 487--490.

    \bibitem{Hend} L.F. Henderson, The statistics of crowd fluids, {\it Nature} \textbf{229} (1971) 381--383.

    \bibitem{Hoog1} S.P. Hoogendoorn, P.H.L. Bovy, Dynamic user-optimal assignment in continuous time and space, {\it Transport. Res. Part B} \textbf{38} (2004) 571--592.
    
    \bibitem{Hoog2} S.P. Hoogendoorn, P.H.L. Bovy, Pedestrian route-choice and activity scheduling theory and models,  {\it Transport. Res. Part B} \textbf{38} (2004) 169--190.
   
    \bibitem{Hug1} R. L. Hughes, A continuum theory for the flow of pedestrian, {\it Transport. Res. Part B} \textbf{36} (2002)  507--535.

   \bibitem{Hug2} R. L. Hughes, The flow of human crowds, {\it Ann. Rev. Fluid Mech.} \textbf{35} (2003) 169--183.

    \bibitem{Ott1} R. Jordan, D. Kinderlehrer, F. Otto, The variational formulation of the Fokker-Planck equation,  {\it SIAM J. Math. Anal.} \textbf{29}(1) (1998) 1--17.

     \bibitem{Kan} L. V. Kantorovich, On the transfer of masses,  {\it Dokl. Akad. Nauk. SSSR} \textbf{37} (1942) 227--229.

   \bibitem{crowd1} B. Maury, J. Venel, Handling of contacts in crowd motion simulations, {\it Traffic and Granular Flow} (Springer, 2007).
       
     \bibitem{crowd2} B. Maury, J. Venel, A mathematical framework for a crowd motion model, {\it C.R. Acad. Sci. Paris, Ser I } \textbf{346} (2008) 1245--1250.

     \bibitem{Mor} J.J. Moreau, Evolution problem associated with a moving convex set in a Hilbert space, {\it J. Differential Equations} \textbf{26}(3) (1977) 346--374.

     \bibitem{Ott2} F. Otto, The geometry of dissipative evolution equations: the porous medium equation, {\it Comm. Partial Differential Equations} \textbf{26}(1--2) (2001) 101--174.

     \bibitem{Pic1} B. Piccoli, A. Tosin, Time-evolving measures and macroscopic modeling of pedestrian flow (2008) to appear.

     \bibitem{Pic2} B. Piccoli, A. Tosin, Pedestrian flows in bounded domains with obstacles,  {\it Contin. Mech. Thermodyn.} (2009) to appear.
    
    \bibitem{Roesch} M. Roesch, Contribution \`a la description g\'eom\'etrique des fluides parfaits incompressibles, {\it PhD Thesis, Universit\'e de Paris 6}.

   \bibitem{crowd3} J. Venel, Integrating strategies in numerical modelling of crowd motion, {\it Pedestrian and Evacuation Dynamics} (Springer, 2008).
       
    
   \bibitem{Vil1} C. Villani, Topics in optimal transportation, {\it Grad. Stud. Math.} \textbf{58} (AMS, Providence 2003).
   
   \bibitem{Vil2} C. Villani, Optimal transport, old and new, {\it Grundlehren der mathematischen Wissenschaften} \textbf{338} (2009).


\end{thebibliography}
\end{document}